%% file: Multdim_int_gap.tex
\RequirePackage[l2tabu,orthodox]{nag} 
\documentclass[a4paper, 12pt, oneside]{article}

\input{packages}

\newcommand{\paperTitle}{High-energy homogenization of a multidimensional nonstationary Schr\"{o}dinger equation}
\newcommand{\paperAuthor}{Mark Dorodnyi}
\newcommand{\paperGrant}{This work was performed at the Saint Petersburg Leonhard Euler International Mathematical Institute and supported by the Ministry of Science and Higher Education of the Russian Federation (agreement №~075-15-2022-287).}
\newcommand{\paperAdress}{St. Petersburg State University, Universitetskaya emb.~7/9, 199034 St.~Petersburg, Russia; e-mail: \texttt{mdorodni@yandex.ru}.}

\providecommand{\keywords}
{periodic differential operators, Schr\"{o}dinger-type equations, spectral bands, homogenization, effective operator, operator error estimates}

\input{newnames}

\input{style}

\begin{document}
	\allowdisplaybreaks
	
\begin{onehalfspacing}
	\begin{center}
		\textbf{\LARGE \paperTitle\textsuperscript{\,}\footnote{\,\,\paperGrant}}
		\\
		\bigskip
		{\large \paperAuthor\textsuperscript{\,}\footnote{\,\,\paperAdress}}
	\end{center}
\end{onehalfspacing}

\begin{abstract}
	\noindent In $L_2(\mathbb{R}^d)$, we consider an elliptic differential operator $\mathcal{A}_\varepsilon = - \operatorname{div} g(\mathbf{x}/\varepsilon) \nabla + \varepsilon^{-2} V(\mathbf{x}/\varepsilon)$, $ \varepsilon > 0$, with periodic coefficients. For the nonstationary Schr\"{o}dinger equation with the Hamiltonian $\mathcal{A}_\varepsilon$, analogs of homogenization problems related to an arbitrary point of the dispersion relation of the operator $\mathcal{A}_1$ are studied (the so called high-energy homogenization).
	For the solutions of the Cauchy problems for these equations with special initial data, approximations in $L_2(\mathbb{R}^d)$-norm for small $\varepsilon$ are obtained.
\end{abstract}


\input{intro}
\input{part1}

\input{biblio}
\end{document}

%% file: packages.tex
\usepackage[utf8]{inputenc}			

\usepackage{cmap}   

\usepackage[english]{babel}	
\usepackage[warn]{mathtext}
\usepackage[T2A]{fontenc}		

\usepackage{comment}    

\usepackage{geometry}   

\usepackage[inline]{enumitem}
\usepackage{indentfirst}
\usepackage{amsthm,amsmath,amscd}   
\usepackage{amsfonts,amssymb}       
\usepackage{mathtools}              
\usepackage{xfrac}                  

\usepackage{esint}

\usepackage{mathrsfs}				

\usepackage{mathrsfs}

\usepackage[inline]{enumitem}		
\usepackage{indentfirst}			

\usepackage{textcomp}

\usepackage[dvipsnames, table, hyperref]{xcolor} 

\IfFileExists{impnattypo.sty}{
	\usepackage[hyphenation, lastparline]{impnattypo}
}{}

\usepackage{microtype}

\usepackage[noadjust]{cite}

\usepackage{hyperref}

\usepackage{titlesec}

%% file: newnames.tex
\DeclareMathOperator{\Dom}{Dom}

\DeclareMathOperator{\Ran}{Ran}
\DeclareMathOperator{\Ker}{Ker}

\DeclareMathOperator{\spec}{spec}

\theoremstyle{plain}
\newtheorem{thrm}{Theorem}[section]

\newtheorem{lemma}[thrm]{Lemma}

\newtheorem{remark}[thrm]{Remark}

\theoremstyle{definition}

%% file: style.tex
\definecolor{linkcolor}{rgb}{0.9,0,0}
\definecolor{citecolor}{rgb}{0,0.6,0}
\definecolor{urlcolor}{rgb}{0,0,1}

\renewcommand{\le}{\leqslant}
\renewcommand{\ge}{\geqslant}

\hypersetup{
	unicode=true,
	linktocpage=true,           
	plainpages=false,           
	colorlinks,                 
	linkcolor={linkcolor},      
	citecolor={citecolor},      
	urlcolor={urlcolor},        
	pdftitle={\paperTitle},    
	pdfauthor={\paperAuthor},  
	pdfkeywords={\keywords},    
	pdflang={ru},
}

\vfuzz \hfuzz
\titleformat{\section}[hang]{ \scshape\bfseries}{\thesection.}{0.5ex}{\centering}[]
\titleformat{\subsection}[runin]{\bfseries}{\thesubsection.}{0.4ex}{}[.]

\usepackage{setspace}
\numberwithin{equation}{section}

%% file: intro.tex
\section*{Introduction}
\addcontentsline{toc}{section}{Introduction}

\subsection{Periodic homogenization}
The study of the wave propagation in periodic structures is of significant interest both for applications and from the
theoretical point of view. Direct numerical simulations of such processes may be difficult. One of the approaches to study these problems is application of homogenization theory. The aim of homogenization is to describe the macroscopic properties of inhomogeneous media by taking into account the properties of the microscopic structure. An extensive literature is devoted to homogenization problems. First of all, we mention the books~\cite{BaPa, BeLP, ZhKO}.

Let us discuss a typical problem of homogenization theory. Let $\Gamma$ be a lattice in $\mathbb{R}^d$, and let $\Omega$ be the cell of $\Gamma$. For any $\Gamma$-periodic function $F(\mathbf{x})$, we denote $F^\varepsilon (\mathbf{x}) \coloneqq F(\varepsilon^{-1} \mathbf{x})$, where $\varepsilon > 0$ is a (small) parameter. In $L_2(\mathbb{R}^d)$,  consider a differential operator~(DO) formally given by
\begin{equation}
	\label{intro_hatA_eps}
	\widehat{\mathcal{A}}_\varepsilon = - \operatorname{div} g^\varepsilon (\mathbf{x}) \nabla,
\end{equation}
where $g(\mathbf{x})$ is a Hermitian $\Gamma$\nobreakdash-periodic $(d \times d)$\nobreakdash-matrix-valued function, bounded and positive definite. Operator~\eqref{intro_hatA_eps} models the simplest cases of microinhomogeneous media with $\varepsilon \Gamma$-periodic structure. Let $u_\varepsilon(\mathbf{x})$ be a (weak) solution of the elliptic equation
\begin{equation}
	\label{intro_elliptic_eq}
	- \operatorname{div} g^\varepsilon(\mathbf{x}) \nabla u_\varepsilon(\mathbf{x}) + u_\varepsilon(\mathbf{x}) = f(\mathbf{x}),
\end{equation} 
where $f \in L_2(\mathbb{R}^d)$. For $\varepsilon \to 0$, the solution $u_\varepsilon$ converges to the solution $u_0$ of the "homogenized" equation:
\begin{equation}
	\label{intro_elliptic_eq_homog}
	- \operatorname{div} g^0 \nabla u_0(\mathbf{x}) + u_0(\mathbf{x}) = f(\mathbf{x}).
\end{equation} 
The operator $\widehat{\mathcal{A}}^{\mathrm{hom}}= - \operatorname{div} g^0 \nabla$ is called the effective operator for $\widehat{\mathcal{A}}_\varepsilon$. The matrix $g^0$ is determined by a well-known procedure (see., e.g.,~\cite[Chapter~2, \S\,3]{BaPa}, \cite[Chapter~3, \S\,1]{BSu2003}) that requires solving an auxiliary boundary value problem on the cell $\Omega$. Besides finding the effective coefficients, the following questions are of great interest. \emph{What is the type of convergence $u_\varepsilon \to u_0$? What is an estimate for $u_\varepsilon - u_0$?}

\subsection{Operator error estimates in homogenization}
M.~Birman and T.~Suslina (see~\cite{BSu2003}) suggested the operator-theoretic (spectral) approach to homogenization problems in $\mathbb{R}^d$, based on the scaling transformation, the Floquet--Bloch theory, and the analytic perturbation theory. 

Let $u_\varepsilon$ be the solution of equation~\eqref{intro_elliptic_eq}, and let $u_0$ be the solution of equation~\eqref{intro_elliptic_eq_homog}. In~\cite{BSu2003},  it was proved that
\begin{equation}
	\label{intro_BSu2003_est_eq}
	\| u_\varepsilon - u_0 \|_{L_2(\mathbb{R}^d)} \le C \varepsilon \|f\|_{L_2(\mathbb{R}^d)}.
\end{equation}
Since $u_\varepsilon = (\widehat{\mathcal{A}}_\varepsilon + I)^{-1} f$ and $u_0 = (\widehat{\mathcal{A}}^{\mathrm{hom}}+ I)^{-1} f$, estimate~\eqref{intro_BSu2003_est_eq} can be rewritten in operator terms:
\begin{equation}
	\label{intro_BSu2003_est}
	\| (\widehat{\mathcal{A}}_\varepsilon + I)^{-1} - (\widehat{\mathcal{A}}^{\mathrm{hom}}+ I)^{-1} \|_{L_2(\mathbb{R}^d) \to L_2(\mathbb{R}^d)} \le C \varepsilon.
\end{equation} 
Parabolic equations were studied in~\cite{Su2004, Su2007}. In operator terms, the following approximation for the parabolic semigroup $e^{-\tau \widehat{\mathcal{A}}_\varepsilon}$, $\tau > 0$, was obtained:
\begin{equation}
	\label{intro_Su2004_est}
	\| e^{-\tau \widehat{\mathcal{A}}_{\varepsilon}} - e^{-\tau \widehat{\mathcal{A}}^{\mathrm{hom}}} \|_{L_2(\mathbb{R}^d) \to L_2(\mathbb{R}^d)} \le C \varepsilon (\tau + \varepsilon^2)^{-1/2}, \qquad \tau > 0.
\end{equation}
Estimates~\eqref{intro_BSu2003_est}, \eqref{intro_Su2004_est} are order-sharp; the constants $C$ are controlled explicitly in terms of the problem data. These estimates are called~\emph{operator error estimates} in homogenization. More accurate approximations for the resolvent and the exponential  with correctors taken into account were found in~\cite{BSu2005, BSu2006, V, Su2010}.

A different approach to operator error estimates (the so called “shift method”) for the elliptic and parabolic problems was suggested by V.~Zhikov and S.~Pastukhova in the papers~\cite{Zh2006, ZhPas2005, ZhPas2006}.  See also the survey~\cite{ZhPas2016}.

The situation with homogenization of nonstationary Schr\"{o}dinger-type equations and hyperbolic equations is quite different. The papers~\cite{BSu2008, Su2017, DSu2018, M2021, DSu2020, D2021} were devoted to such problems. In operator terms, the behavior of the operator-functions $e^{-i \tau \widehat{\mathcal{A}}_{\varepsilon}}$ and $\cos(\tau \widehat{\mathcal{A}}_{\varepsilon}^{1/2})$, $\widehat{\mathcal{A}}_{\varepsilon}^{-1/2} \sin(\tau \widehat{\mathcal{A}}_{\varepsilon}^{1/2})$ (where $\tau \in \mathbb{R}$) for small $\varepsilon$ was studied. For these operator-functions, it is impossible to obtain approximations in the operator norm on $L_2 (\mathbb{R}^d)$, and we are forced to consider the norm of operators acting from the Sobolev space $H^q (\mathbb{R}^d)$ (with a suitable $q$) to $L_2 (\mathbb{R}^d)$. In~\cite{BSu2008}, the following sharp-order estimates were proved:
\begin{align}
	\label{intro_exp_est}
	\| e^{-i \tau \widehat{\mathcal{A}}_{\varepsilon}} - e^{-i \tau \widehat{\mathcal{A}}^{\mathrm{hom}}} \|_{H^3 (\mathbb{R}^d) \to L_2 (\mathbb{R}^d)} &\le C (1 + |\tau|)\varepsilon, \\
	\label{intro_cos_est}
	\| \cos(\tau \widehat{\mathcal{A}}_{\varepsilon}^{1/2}) - \cos(\tau (\widehat{\mathcal{A}}^{\mathrm{hom}})^{1/2}) \|_{H^2 (\mathbb{R}^d) \to L_2 (\mathbb{R}^d)} &\le  C (1 + |\tau|)\varepsilon.
\end{align}
In~\cite{M2021}, the result for the operator $\widehat{\mathcal{A}}_{\varepsilon}^{-1/2} \sin(\tau \widehat{\mathcal{A}}_{\varepsilon}^{1/2})$ was obtained:
\begin{equation}
	\label{intro_sin_est}
	\| \widehat{\mathcal{A}}_{\varepsilon}^{-1/2} \sin(\tau \widehat{\mathcal{A}}_{\varepsilon}^{1/2}) - (\widehat{\mathcal{A}}^{\mathrm{hom}})^{-1/2} \sin(\tau (\widehat{\mathcal{A}}^{\mathrm{hom}})^{1/2}) \|_{H^1 (\mathbb{R}^d) \to L_2 (\mathbb{R}^d)} \le C (1 + |\tau|)\varepsilon.
\end{equation}
Moreover, in~\cite{M2021}, an approximation of the operator $\widehat{\mathcal{A}}_{\varepsilon}^{-1/2} \sin(\tau \widehat{\mathcal{A}}_{\varepsilon}^{1/2})$ for a fixed $\tau$ in the \mbox{$(H^2 \to H^1)$}\nobreakdash-norm with error of order $O (\varepsilon)$ (with a corrector taken into account) was obtained. Next, in~\cite{Su2017, DSu2018, DSu2020, D2021}, it was shown that these results are sharp with respect to the norm type as well as with respect to the dependence on $\tau$ (for large $\tau$). On the other hand, it was shown  that under some additional assumptions (e.g., if the matrix $g(\mathbf{x})$ has real entries) estimates~\eqref{intro_exp_est}--\eqref{intro_sin_est} can be improved:
\begin{alignat}{2}
	\notag
	\| e^{-i \tau \widehat{\mathcal{A}}_{\varepsilon}} - e^{-i \tau \widehat{\mathcal{A}}^{\mathrm{hom}}} &\|_{H^2 (\mathbb{R}^d) \to L_2 (\mathbb{R}^d)} &  &\le C (1 + |\tau|^{1/2})\varepsilon, 
	\\
	\label{intro_cos_improved_est}
	\| \cos(\tau \widehat{\mathcal{A}}_{\varepsilon}^{1/2}) - \cos(\tau (\widehat{\mathcal{A}}^{\mathrm{hom}})^{1/2}) &\|_{H^{3/2} (\mathbb{R}^d) \to L_2 (\mathbb{R}^d)} & &\le  C (1 + |\tau|^{1/2})\varepsilon,
	\\
	\label{intro_sin_improved_est}
	\| \widehat{\mathcal{A}}_{\varepsilon}^{-1/2} \sin(\tau \widehat{\mathcal{A}}_{\varepsilon}^{1/2}) - (\widehat{\mathcal{A}}^{\mathrm{hom}})^{-1/2} \sin(\tau (\widehat{\mathcal{A}}^{\mathrm{hom}})^{1/2}) &\|_{H^{1/2} (\mathbb{R}^d) \to L_2 (\mathbb{R}^d)} & &\le C (1 + |\tau|^{1/2})\varepsilon.
\end{alignat}

Note that in~\cite{BSu2003, Su2004,  Su2007, BSu2005, BSu2006, V, Su2010, BSu2008, Su2017, DSu2018, M2021, DSu2020, D2021,Su2022a,Su2022b}, a much broader class of operators than~\eqref{intro_hatA_eps} (including matrix DOs) was studied. In particular, operators of the form
\begin{equation}
	\label{intro_A_eps}
	\mathcal{A}_{\varepsilon} = - \operatorname{div} \check{g}^{\varepsilon}(\mathbf{x}) \nabla + \varepsilon^{-2}V^{\varepsilon}(\mathbf{x})
\end{equation}
were considered. Here $\check{g}(\mathbf{x})$ is a $\Gamma$-periodic positive definite and bounded $(d \times d)$-matrix-valued function with real entries, $V(\mathbf{x})$ is a $\Gamma$-periodic real-valued function, $V \in L_p(\Omega)$ with a suitable $p$ (and it is assumed that $\inf \spec \mathcal{A}_1 = 0$). 
For operator~\eqref{intro_A_eps}, it is impossible to find an operator $\mathcal{A}^{\mathrm{hom}}$ with constant coefficients such that the corresponding operator-functions converge to the operator-functions of $\mathcal{A}^{\mathrm{hom}}$. However, some approximations can be found if we "border" operator-functions of $\widehat{\mathcal{A}}^{\mathrm{hom}}$ by appropriate rapidly oscillating factors. In particular, an analog of~\eqref{intro_BSu2003_est} is as follows:
\begin{equation*}
	\| (\mathcal{A}_\varepsilon + I)^{-1} - [\omega^\varepsilon](\widehat{\mathcal{A}}^{\mathrm{hom}}+ I)^{-1}[\omega^\varepsilon] \|_{L_2(\mathbb{R}^d) \to L_2(\mathbb{R}^d)} \le C \varepsilon,
\end{equation*} 
where $\omega (\mathbf{x})$ is a positive $\Gamma$-periodic solution of the equation
\begin{equation*}
	- \operatorname{div} \check{g}(\mathbf{x}) \nabla \omega (\mathbf{x}) + V(\mathbf{x}) \omega (\mathbf{x}) = 0
\end{equation*}
satisfying the normalization condition $\| \omega \|_{L_2(\Omega)}^2 = | \Omega |$, and $\widehat{\mathcal{A}}^{\mathrm{hom}}$ is the effective operator for operator~\eqref{intro_hatA_eps} with the matrix $g(\mathbf{x}) = \check{g} (\mathbf{x}) \omega^2 (\mathbf{x})$.

Let us explain the method using the example of operator~\eqref{intro_hatA_eps}. The scaling transformation reduces investigation of the behavior of the operator $(\widehat{\mathcal{A}}_\varepsilon + I)^{-1}$, $\varepsilon \to 0$, to studying the operator $(\widehat{\mathcal{A}} + \varepsilon^2 I)^{-1}$, where $\widehat{\mathcal{A}} = \widehat{\mathcal{A}}_1 = - \operatorname{div} g(\mathbf{x}) \nabla$. Next, by the Floquet--Bloch theory, the operator~$\widehat{\mathcal{A}}$ expands in the direct integral of the operators $\widehat{\mathcal{A}}(\mathbf{k})$ acting in the space $L_2 (\Omega)$. The operator $\widehat{\mathcal{A}}(\mathbf{k})$ is defined by the differential expression $-\operatorname{div}_\mathbf{k} g(\mathbf{x}) \nabla_\mathbf{k}$, where $\nabla_\mathbf{k} = \nabla + i \mathbf{k}$, $\operatorname{div}_\mathbf{k} = \operatorname{div} + i \langle \mathbf{k}, \cdot \rangle$, with periodic boundary conditions. The spectrum of the operator $\widehat{\mathcal{A}}(\mathbf{k})$ is discrete.
It turns out that the behavior of the resolvent $(\widehat{\mathcal{A}} + \varepsilon^2 I)^{-1}$ can be described in terms of the threshold characteristics of $\widehat{\mathcal{A}}$ at the edge of the spectrum, i.e., it is sufficient to know the spectral decomposition of $\widehat{\mathcal{A}}$ only near the lower edge of the spectrum. In particular, the effective matrix $g^0$ is a Hessian of the first band function $E_1(\mathbf{k})$ at the point $\mathbf{k} = 0$.

Finally, we mention the recent paper~\cite{LiSh2021}, where the authors investigated the problem of convergence rates for a solution of the initial-Dirichlet boundary value problem for a wave equation; analogs of estimates~\eqref{intro_cos_improved_est}, \eqref{intro_sin_improved_est} as well as results with the Dirichlet corrector were obtained.

\subsection{High-frequency homogenization}
As stated above, only a small neighborhood of the bottom of the spectrum (i.e., waves with low frequencies) contributes to homogenization. However, we can consider problems of wave propagation when the frequency is proportional to $\varepsilon^{-1}$ or $\varepsilon^{-2}$ (the high-frequency mode). In this case, even the leading order of the asymptotics oscillates rapidly. These problems were studied in~\cite[Chapter~4]{BeLP} using WKB-ansatz.

Traditional methods of homogenization theory, related to asymptotic expansions in two scales, were applied to these problems in~\cite{CrKaPi2010, HaMiCr2016}. We also cite the paper~\cite{Ce_etal_2015}, where application of the results of~\cite{CrKaPi2010} to photonic crystals was considered. In~\cite{CrKaPi2010}, an asymptotic expansion for solutions of the equation
\begin{equation*}
	\operatorname{div} g^\varepsilon(\mathbf{x}) \nabla u_\varepsilon(\mathbf{x}) + \nu^2 \rho^\varepsilon(\mathbf{x}) u_\varepsilon(\mathbf{x}) = 0,
\end{equation*}
which are perturbations of the standing waves, was obtained (the functions $g(\mathbf{x})$, $\rho(\mathbf{x})$ were supposed to be sufficiently smooth and $\Gamma$-periodic). In~\cite{HaMiCr2016}, a similar problem for travelling waves was considered.

For a nonstationary Schr\"{o}dinger equation results of this kind are called effective mass theorems (see, e.g., the course~\cite{A2008} and references therein). In the paper~\cite{APi2004}, homogenization of the Cauchy problem for a nonstationary Schr\"{o}dinger equation with well-prepared initial data concentrating on a Bloch eigenfunction was studied using techniques of two-scale convergence and suitable oscillating test functions; a rigorous derivation of effective mass theorems was obtained (in terms of the strong two-scale convergence).
In~\cite{BaBenAb2011}, the effective mass approximation and the $k\cdot p$ multi-band models, well known in solid-state physics, were discussed. Such homogenization asymptotics were investigated by using the envelope-function decomposition. These models were proved to be close (in the strong sense) to the exact dynamics. Moreover, the position density was proved to converge weakly to its effective mass approximation.

Finally, we also mention the papers~\cite{KuRa2012,KhKuRa2017}, where asymptotics of Green’s function for different values of the spectral parameter has been studied.

Now, let us discuss error estimates for high-frequency homogenization. This topic has been studied in~\cite{B03,SuKh2009,MishSlSu2022,AkhAkSlSu2021,D2022} in the one-dimensional case ($d=1$) and in~\cite{BSu2004, SuKh2011, Mish2022} in the case of arbitrary dimension $d$. It is well-known that the spectrum of $\mathcal{A}$ has a band structure and may have gaps. For the sake of simplicity, we consider the case where $d=1$ and $\Gamma = \mathbb{Z}$;  in this case we shall use the notation $A_\varepsilon$ for operator~\eqref{intro_A_eps}. Let $\sigma > 0$ be a (non-degenerate) left edge of a band with an odd number ($\ge 3$) in the spectrum of the operator $A = A_1$. Then for $A_\varepsilon$, this edge "moves" to the point $\varepsilon^{-2} \sigma$ (to the high-frequency (high-energy) region). Instead of~\eqref{intro_elliptic_eq}, we consider the
equation
\begin{equation}
	\label{intro_int_edge_elliptic_eq_est}
	- \frac{d}{dx} g^\varepsilon(x) \frac{d}{dx} u_\varepsilon(x) - (\varepsilon^{-2} \sigma - \varkappa^2) u_\varepsilon(x) = f(x),
\end{equation} 
where $f \in L_2(\mathbb{R})$. It is supposed that $\varkappa > 0$ is such that the point $\varepsilon^{-2} \sigma - \varkappa^2$ belongs to the gap in the spectrum of the operator $A_\varepsilon$. Similarly to~\eqref{intro_BSu2003_est}, the question is reduced to studying the operator $(A_\varepsilon - (\varepsilon^{-2} \sigma - \varkappa^2)I)^{-1}$. In~\cite{B03}, the following result was proved:
\begin{equation}
	\label{intro_int_gap_est}
	\| (A_\varepsilon - (\varepsilon^{-2} \sigma - \varkappa^2)I)^{-1} - [\varphi_\sigma^\varepsilon](A_\sigma^{\mathrm{hom}} + \varkappa^2 I)^{-1}[\varphi_\sigma^\varepsilon] \|_{L_2(\mathbb{R}) \to L_2(\mathbb{R})} \le C \varepsilon.
\end{equation} 
Here $A_\sigma^{\mathrm{hom}} = -b_\sigma \frac{d^2}{dx^2}$ is the corresponding effective operator, $b_\sigma > 0$ is the coefficient in the asymptotics of the band function $E(k)$ corresponding to the
band for which $\sigma$ is the left edge: $E(k) \sim \sigma + b_\sigma k^2$, $k \sim 0$; and $\varphi_\sigma$ is a real-valued periodic solution of the equation $A \varphi_\sigma = \sigma \varphi_\sigma$, normalized in $L_2(0,1)$. Consequently, the possibility of homogenization for equation~\eqref{intro_int_edge_elliptic_eq_est} is a threshold effect near the edge of an internal gap.

Estimate~\eqref{intro_int_gap_est} was obtained in~\cite{B03} in the case where $V(x) = 0$. In~\cite{BSu2004}, an analog of estimate~\eqref{intro_int_gap_est} was proved for operators~\eqref{intro_A_eps} in arbitrary dimension $d \ge 1$. More accurate approximations with correctors were obtained in~\cite{SuKh2009, MishSlSu2022, SuKh2011}. 

Parabolic equations in the one-dimensional case were studied in~\cite{AkhAkSlSu2021}. It was proved that
\begin{equation*}
	\| e^{-\tau A_{\varepsilon}} \mathcal{E}_{A_\varepsilon}[\varepsilon^{-2} \sigma, \infty) - e^{-\tau \sigma/\varepsilon^2} [\varphi_\sigma^\varepsilon]e^{-t A_\sigma^{\mathrm{hom}}}[\varphi_\sigma^\varepsilon] \|_{L_2(\mathbb{R}) \to L_2(\mathbb{R})} \le C e^{-\tau \sigma/\varepsilon^2} \varepsilon (\tau + \varepsilon^2)^{-1/2}, \quad \tau > 0,
\end{equation*}
and a more accurate approximation with a corrector was found. Here $\mathcal{E}_{A_\varepsilon}[\varepsilon^{-2} \sigma, \infty)$ is the spectral projection of the operator $A_\varepsilon$ corresponding to the
interval $[\varepsilon^{-2} \sigma, \infty)$. The generalization of this result for the case of arbitrary dimension was obtained in~\cite{Mish2022}.

In the paper~\cite{D2022}, operator error estimates for high-frequency homogenization of nonstationary Schr\"{o}dinger equations and hyperbolic equations in the one-dimensional case ($d=1$) were studied. Let $f_1, f_2 \in L_2(\mathbb{R})$. Consider the Cauchy problems
\begin{equation}
	\label{intro_Cauchy_problem_Schrod_hyperb}
	\left\{
	\begin{aligned}
		&i \frac{\partial}{\partial \tau} u_\varepsilon (x,\tau) = (A_\varepsilon u_\varepsilon)(x,\tau) ,\\
		&u_\varepsilon (x,0) = (\Upsilon_\varepsilon f_1)(x),
	\end{aligned}
	\right.
	\qquad 
	\left\{
	\begin{aligned}
		&\frac{\partial^2}{\partial \tau^2}  v_\varepsilon (x,\tau) = - (A_\varepsilon v_\varepsilon)(x,\tau) + \varepsilon^{-2} \sigma v_\varepsilon(x,\tau),\\
		&v_\varepsilon (x,0) = (\Upsilon_\varepsilon f_1)(x), \; (\partial_\tau v_\varepsilon) (x,0) = (\Upsilon_\varepsilon f_2)(x),
	\end{aligned}
	\right.
\end{equation} 
where
\begin{equation*}
	(\Upsilon_\varepsilon f)(x) \coloneqq (2 \pi)^{-1/2} \int_\mathbb{R} (\Phi f) (k) \sum_{j=s}^{\infty} e^{ikx} \varphi_j(x/\varepsilon, \varepsilon k) \chi_{\widetilde{\Omega}_{j-s+1}} (\varepsilon k) \, dk.
\end{equation*}
Here $\{e^{ikx} \varphi_j(x, k)\}_{j=s}^{\infty}$ are the Bloch waves corresponding to the bands with the numbers $j \ge s$;
$\widetilde{\Omega}_j = (-j \pi, -(j-1)\pi] \cup ((j-1)\pi, j \pi ]$, $j \in \mathbb{N}$, are the Brillouin zones. The initial data of problems~\eqref{intro_Cauchy_problem_Schrod_hyperb} are superpositions of the Bloch waves with the amplitudes, which are equal to the Fourier images $(\Phi f_1) (k)$, $(\Phi f_2) (k)$ of the functions $f_1(x)$, $f_2(x)$, and belong to the subspace $\mathcal{E}_{A_\varepsilon}[\varepsilon^{-2} \sigma, \infty) L_2(\mathbb{R})$. The following approximations were found:
\begin{align}
	\label{intro_Schrod_est}
	&\| u_\varepsilon(\cdot, \tau) - e^{-i \tau \varepsilon^{-2} \sigma} \varphi_\sigma^{\varepsilon}  u_0(\cdot, \tau)\|_{L_2(\mathbb{R})} \le C (1+|\tau|^{1/2}) \varepsilon \|f_1\|_{H^2(\mathbb{R})}, \qquad  f_1 \in H^2(\mathbb{R}),
	\\
	\label{intro_hyperb_est}
	&\begin{multlined}[c][0.9\textwidth]
		\| v_\varepsilon(\cdot, \tau) - \varphi_\sigma^{\varepsilon}  v_0(\cdot, \tau)\|_{L_2(\mathbb{R})} \le C (1+|\tau|^{1/2}) \varepsilon (\|f_1\|_{H^{3/2}(\mathbb{R})} + \|f_2\|_{H^{1/2}(\mathbb{R})}), \\ f_1 \in H^{3/2}(\mathbb{R}), \; f_2 \in H^{1/2}(\mathbb{R}).
	\end{multlined}	
\end{align}
Here $u_0$ and $v_0$ are the solutions of the effective problems
\begin{equation*}
	\left\{
	\begin{aligned}
		&i \frac{\partial}{\partial \tau} u_0 (x,\tau) = (A^\textrm{hom}_\sigma u_0)(x,\tau) ,\\
		&u_0 (x,0) = f_1 (x),
	\end{aligned}
	\right.
	\qquad
	\left\{
	\begin{aligned}
		&\frac{\partial^2}{\partial \tau^2} v_0 (x,\tau) = - (A^\textrm{hom}_\sigma v_0)(x,\tau) ,\\
		&v_0 (x,0) = f_1 (x), \; (\partial_\tau v_0) (x,0) = f_2 (x).
	\end{aligned}
	\right.
\end{equation*}
Note that estimates~\eqref{intro_Schrod_est}, \eqref{intro_hyperb_est} can be formulated in operator terms; see~\cite[(6.6), (6.21)--(6.23)]{D2022}.

\subsection{Main results}
In the present paper, we study error estimates for high-energy homogeniza\-tion of nonstationary Schr\"{o}dinger equations in the case of an arbitrary dimension. Let $(\mathbf{k}^\circ, \lambda_0)$ be an arbitrary point of the dispersion relation $\mathfrak{B}_\mathcal{A}$ of the operator $\mathcal{A} \coloneqq \mathcal{A}_1$. In particular, it may be an extremum of a band function or a point where two branches of the dispersion relation meet (they often form the so-called Dirac cone, see~\cite[Sec.~5.10]{Ku2016}). Let $\{e^{i\left\langle \mathbf{k}^\circ, \mathbf{x} \right\rangle} \varsigma_j (\mathbf{k}^\circ, \mathbf{x}) \}_{j=1}^n$ be corresponding Bloch waves; we suppose that $\bigl(\varsigma_j (\mathbf{k}^\circ, \cdot), \varsigma_k (\mathbf{k}^\circ, \cdot)\bigr)_{L_2(\Omega)} = \delta_{jk}$. We are interested in the behavior of the solutions $u_{j,\varepsilon} (\mathbf{x}, \tau)$, $\mathbf{x} \in \mathbb{R}^d$, $\tau \in \mathbb{R}$, $j = 1, \ldots,n$, of the following Cauchy problems for the nonstationary Schr\"{o}dinger equation
	\begin{equation*}
		\left\{
		\begin{aligned}
			&i \frac{\partial}{\partial \tau} u_{j,\varepsilon} (\mathbf{x},\tau) = (\mathcal{A}_\varepsilon u_{j,\varepsilon})(\mathbf{x},\tau),
			\\
			&u_{j,\varepsilon} (\mathbf{x},0) = e^{i \varepsilon^{-1} \left\langle \mathbf{k}^\circ, \mathbf{x} \right\rangle } \varsigma_j^\varepsilon(\mathbf{k}^\circ,\mathbf{x}) f_j(\mathbf{x}),
		\end{aligned}
		\right.
	\end{equation*} 
as $\varepsilon \to 0$, where $\varsigma_j^\varepsilon(\mathbf{k}^\circ,\mathbf{x}) \coloneqq \varsigma_j(\mathbf{k}^\circ,\mathbf{x}/\varepsilon)$, and $f_j(\mathbf{x})$, $j=1,\ldots,n$, are given functions. \emph{Main results of the paper} are the following estimates:
\begin{equation*}
	\| u_{j,\varepsilon} (\cdot, \tau) - u^{\mathrm{eff}}_{j,\varepsilon} (\cdot, \tau) \|_{L_2(\mathbb{R}^d)} \le \mathcal{C} (1+|\tau|) \varepsilon \|f_j\|_{H^3(\mathbb{R}^d)}, \qquad j = 1, \ldots,n,
\end{equation*}
where
\begin{equation*}
	u^{\mathrm{eff}}_{j,\varepsilon} (\mathbf{x},\tau) \coloneqq e^{i \varepsilon^{-1} \left\langle \mathbf{k}^\circ, \mathbf{x} \right\rangle } \sum_{l=1}^{n} \varsigma_l^\varepsilon (\mathbf{k}^\circ,\mathbf{x}) v^{\mathrm{eff}}_{jl, \varepsilon} (\mathbf{x},\tau) 
\end{equation*}
and $\mathbf{v}^{\mathrm{eff}}_{j, \varepsilon} (\mathbf{x},\tau) = (v^{\mathrm{eff}}_{j1, \varepsilon} (\mathbf{x},\tau), \ldots, v^{\mathrm{eff}}_{jn, \varepsilon} (\mathbf{x},\tau))^\mathrm{t}$ is the solution of the "effective" system
\begin{equation*}
	\left\{
	\begin{aligned}
		&i \frac{\partial}{\partial \tau} \mathbf{v}^{\mathrm{eff}}_{j, \varepsilon} (\mathbf{x},\tau) = \mathcal{A}^{\mathrm{eff}}_\varepsilon \mathbf{v}^{\mathrm{eff}}_{j, \varepsilon} (\mathbf{x},\tau) ,\\
		&\mathbf{v}^{\mathrm{eff}}_{j, \varepsilon} (\mathbf{x},0) = f_j (\mathbf{x}) \mathbf{e}_j.
	\end{aligned}
	\right.
\end{equation*}
Here $\mathcal{A}^{\mathrm{eff}}_\varepsilon$ is \emph{an effective operator} with \emph{constant coefficients} (its definition is given below in~\eqref{A^eff_def}, \eqref{A^eff_lp_def}), and $\mathbf{e}_j$ is the element of the canonical basis in $\mathbb{C}^n$.

\subsection{Plan of the paper}
The paper consists of Introduction and four more sections. In Sec.~\ref{ch1}, a precise definition of the operator $\mathcal{A}$ is given, its factorization is described. Next, in Sec.~\ref{ch2}, we describe a spectral expansion of the operator $\mathcal{A}$ (partial diagonalization via the Gelfand transformation). Then, in Sec.~\ref{ch3}, spectral approximations for the operator $\mathcal{A}$ in some neighbourhood of the point $(\mathbf{k}^\circ, \lambda_0) \in \mathfrak{B}_\mathcal{A}$ are obtained, and also the effective characteristics are calculated. Finally, in Sec.~\ref{ch4}, we formulate and prove the main result of the paper.

\subsection{Notation}
Let $\mathfrak{H}$ and $\mathfrak{H}_{*}$ be complex separable Hilbert spaces. The symbols $(\cdot, \cdot)_{\mathfrak{H}}$ and $ \| \cdot \|_{\mathfrak{H}}$ denote the inner product and the norm in $\mathfrak{H}$. The symbol $\| \cdot \|_{\mathfrak{H} \to \mathfrak{H}_*}$ stands for the norm of a bounded linear operator from $\mathfrak{H}$ to $\mathfrak{H}_{*}$. Sometimes we omit the indices. By $I = I_{\mathfrak{H}}$ we denote the identity operator in $\mathfrak{H}$. If $A \colon \mathfrak{H} \to \mathfrak{H}_*$ is a linear operator, then $\Dom A$ and $\Ran A$ stand for its domain and range, respectively. If $\mathfrak{N}$ is a subspace in $\mathfrak{H}$, then $\mathfrak{N}^{\perp} \coloneqq \mathfrak{H} \ominus \mathfrak{N}$. If $P$ is the orthogonal projection of $\mathfrak{H}$ onto $\mathfrak{N}$, then $P^{\perp}$ is the orthogonal projection of $\mathfrak{H}$ onto $\mathfrak{N}^{\perp}$. Next, if $A$ is a selfadjoint operator in some Hilbert space, then we use the notation $\spec A$ for the spectrum of $A$.

The symbol $\left< \cdot, \cdot \right>$ stands for the standard inner product in $\mathbb{C}^n$. For $z \in \mathbb{C}$, by $z^*$ we denote the complex conjugate number. If $a$ is an $(m \times n)$-matrix, then $a^\mathrm{t}$ denotes the transpose matrix, and $a^*$ stands for the adjoint $(n \times m)$-matrix. By $\{\mathbf{e}_j\}_{j=1}^n$ we denote the canonical basis in $\mathbb{C}^n$.

The standard $L_p$ classes of functions in a domain $\mathcal{O} \subset \mathbb{R}^d$ are denoted by $L_p (\mathcal{O})$, $1 \le p \le \infty$; $H^q (\mathcal{O})$ are the Sobolev classes of functions in a domain $\mathcal{O} \subset \mathbb{R}^d$ of order $q \in \mathbb{R}$ and integrability index~$2$. If $f$ is a measurable function, then the operator of multiplication by the function $f$ in the space $L_2$ is denoted by the same symbol.

Next, $\mathbf{x} = (x_1, \ldots , x_d) \in \mathbb{R}^d$, $i D_j = \frac{\partial}{\partial x_j}$, $j = 1,\ldots, d$, $\mathbf{D} = -i \nabla = (D_1, \ldots, D_d)$. 

By $\Phi \coloneqq \Phi_{\mathbf{x} \to \mathbf{k}}$ we denote the Fourier transform on $\mathbb{R}^d$ defined on the Schwartz class by the formula
\begin{equation*}
	(\Phi v) (\mathbf{k}) = (2 \pi)^{-d/2} \int_{\mathbb{R}} e^{-i \left\langle \mathbf{k}, \mathbf{x} \right\rangle} v(\mathbf{x}) \, d \mathbf{x}, \qquad v \in \mathcal{S}(\mathbb{R}^d),
\end{equation*}
and extended by continuity up to the unitary mapping $\Phi \colon L_2(\mathbb{R}^d) \to L_2(\mathbb{R}^d)$. For the ball of radius $\varkappa$ centered at $\mathbf{k}' \in \mathbb{R}^d$, we use the notation $\mathrm{B}_{\varkappa}(\mathbf{k}')$. 

\subsection{Acknowledgments} The author is grateful to T.~A.~Suslina for helpful discussions and attention to the work. The author is a Young Russian Mathematics award winner and would like to thank its sponsors and jury.

%% file: part1.tex
\section{The operator $\mathcal{A}$}
\label{ch1}
Let $\Gamma$ be a lattice in $\mathbb{R}^d$ generated by a basis $\mathbf{a}_1, \ldots , \mathbf{a}_d$: 
\begin{equation*}
	\Gamma = \Bigl\{ \mathbf{a} \in \mathbb{R}^d \colon \mathbf{a} = \sum_{j=1}^{d} n^j \mathbf{a}_j, \; n^j \in \mathbb{Z} \Bigr\},
\end{equation*}
and let $\Omega$ be the elementary cell of the lattice $\Gamma$: 
\begin{equation*}
	\Omega \coloneqq \Bigl\{ \mathbf{x} \in \mathbb{R}^d \colon \mathbf{x} = \sum_{j=1}^{d} \xi^j \mathbf{a}_j, \; 0 < \xi^j < 1 \Bigr\}.
\end{equation*}
The basis $\mathbf{b}^1, \ldots , \mathbf{b}^d$ dual to $\mathbf{a}_1, \ldots , \mathbf{a}_d$ is defined by the relations $\left< \mathbf{b}^l, \mathbf{a}_j \right> = 2 \pi \delta^l_j$. This basis generates the lattice $\widetilde \Gamma$, \emph{dual} to the lattice $\Gamma$.
By $\widetilde{\Omega}$ we denote \emph{the central Brillouin zone} of the lattice $\widetilde{\Gamma}$:
\begin{equation*}
	\operatorname{Int} \widetilde{\Omega} = \bigl\{ \mathbf{k} \in \mathbb{R}^d \colon | \mathbf{k} | < | \mathbf{k} - \mathbf{b} |, \; 0 \ne \mathbf{b} \in \widetilde \Gamma \bigr\}.
\end{equation*}
By $\widetilde{H}^1(\Omega)$ we denote the subspace of functions in $H^1(\Omega)$, whose   $\Gamma$\nobreakdash-periodic extension to $\mathbb{R}^d$ belongs to $H^1_{\mathrm{loc}}(\mathbb{R}^d)$.

In $L_2(\mathbb{R}^d)$, $d \ge 1$, we consider a selfadjoint $\Gamma$-periodic Schr\"{o}dinger operator $\mathcal{A}$ generated by the differential expression
\begin{equation}
	\label{A_with_V}
	\mathcal{A} = - \operatorname{div} \check{g} (\mathbf{x}) \nabla + V(\mathbf{x}) = \mathbf{D}^* \check{g}(\mathbf{x}) \mathbf{D} + V(\mathbf{x})
\end{equation}
with metric $\check{g} (\mathbf{x})$ and potential $V(\mathbf{x})$. It is supposed that
\begin{equation}
	\label{g_check_cond}
	\left. 
	\begin{aligned}
		\check{g} \text{ is a measurable symmetric matrix-valued function with real entries,}&\\
		\alpha_0 \mathbf{1} \le \check{g} (\mathbf{x}) \le \alpha_1 \mathbf{1}, \qquad 0 < \alpha_0 \le \alpha_1 < \infty,&
	\end{aligned}	
	\;\right\rbrace 
\end{equation}
and $V (\mathbf{x})$ is a real-valued function such that
\begin{equation*}
	V \in L_q(\Omega), \qquad q > d/2 \quad  \text{for} \quad d \ge 2, \qquad q=1 \quad \text{for} \quad d = 1.
\end{equation*}
The precise definition of the operator $\mathcal{A}$ is given in terms of the semi-bounded closed quadratic form
\begin{equation}
	\label{a_form_1}
	\mathfrak{a}[u,u] = \int_{\mathbb{R}} (\left\langle \check{g} (\mathbf{x}) \mathbf{D} u, \mathbf{D} u \right\rangle  + V(\mathbf{x}) |u(\mathbf{x})|^2) \, d \mathbf{x}, \qquad u \in H^1(\mathbb{R}^d).
\end{equation}
Adding an appropriate constant to $V$, we assume that $\inf \spec\mathcal{A} = 0$. Under this assumption the operator $\mathcal{A}$ admits a convenient factorization (see, e.g., \cite{KiSi1987}, \cite[Ch.~6, Sec.~1.1]{BSu2003}). To describe this factorization, we consider the equation
\begin{equation*}
	\mathbf{D}^* \check{g}(\mathbf{x}) \mathbf{D} \omega (\mathbf{x}) + V(\mathbf{x}) \omega (\mathbf{x}) = 0
\end{equation*}
(which is understood in the weak sense). There exists a (strictly) positive $\Gamma$\nobreakdash-periodic solution $\omega \in \widetilde{H}^1 (\Omega)$ of this equation defined up to a constant factor. This factor can be fixed so that
\begin{equation}
	\label{omega_cond1}
	\| \omega \|^2_{L_2(\Omega)} = |\Omega|.
\end{equation}

It turns out that $\omega \in C^\kappa$ for some $\kappa > 0$. Moreover, the function $\omega$ is a multiplier in $H^1 (\mathbb{R}^d)$ and in $\widetilde{H}^1 (\Omega)$. 
The substitution $u = \omega \psi$ transforms form~\eqref{a_form_1} to the form
\begin{equation}
	\label{a_form_2}
	\mathfrak{a}[u,u] = \int_{\mathbb{R}^d} \omega^2 (\mathbf{x}) \left\langle  \check{g}(\mathbf{x}) \mathbf{D}\psi, \mathbf{D}\psi \right\rangle  \, d\mathbf{x}, \qquad u = \omega \psi, \quad \psi \in H^1(\mathbb{R}^d).
\end{equation}
This yields the factorization
\begin{equation}
	\label{A}
	\mathcal{A} = \omega(\mathbf{x})^{-1} \mathbf{D}^* g(\mathbf{x}) \mathbf{D} \omega(\mathbf{x})^{-1}, \qquad g = \omega^2(\mathbf{x}) \check{g}(\mathbf{x}).
\end{equation}
We take representation~\eqref{A} of the operator $\mathcal{A}$ as the initial definition, i.e., we assume that $\mathcal{A}$ is the operator generated by form~\eqref{a_form_2}, where $\check{g}$ and $\omega$ are $\Gamma$-periodic functions satisfying~\eqref{g_check_cond}, \eqref{omega_cond1} and the conditions $\omega(\mathbf{x}) > 0$; $\omega, \omega^{-1} \in L_\infty$. We can return to representation~\eqref{A_with_V} putting $V = - \omega^{-1} (\mathbf{D}^* \check{g} \mathbf{D} \omega)$. However,
the potential $V$ may be highly singular.

\section{Spectral decomposition of the operator~$\mathcal{A}$}
\label{ch2}
We need to describe the spectrum of the operator~\eqref{A}. For this, let us introduce the objects associated with the spectral resolution of operator~\eqref{A}. Put
\begin{alignat*}{4}
	&\mathcal{H}^1(\mathcal{O}) & &= & \; &\{f \colon \omega^{-1} f \in H^1(\mathcal{O})\}, &\qquad &\text{where} \quad \mathcal{O} = \mathbb{R}^d \;\; \text{or} \;\; \Omega, \\
	&\widetilde{\mathcal{H}}^1(\Omega) & &=  & &\{f \colon \omega^{-1} f \in \widetilde{H}^1(\Omega)\}, & \qquad
	&\| f \|_{\mathcal{H}^1(\mathcal{O})} = \| \omega^{-1} f \|_{H^1(\mathcal{O})}.
\end{alignat*}
In $L_2(\Omega)$, consider the family of quadratic forms
\begin{equation}
	\label{a(k)_form}
	\mathfrak{a}(\mathbf{k})[u,u] = \int_\Omega \left\langle  g (\mathbf{x}) (\mathbf{D} + \mathbf{k}) \omega^{-1} u, (\mathbf{D} + \mathbf{k}) \omega^{-1} u \right\rangle \, d \mathbf{x}, \qquad u \in \widetilde{\mathcal{H}}^1(\Omega), \quad \mathbf{k} \in \mathbb{R}^d.
\end{equation}
The operator generated by form~\eqref{a(k)_form} is denoted by $\mathcal{A} (\mathbf{k})$. Formally, we can write
\begin{equation*}
	\mathcal{A} (\mathbf{k}) = \omega(\mathbf{x})^{-1} (\mathbf{D}+ \mathbf{k})^* g(\mathbf{x}) (\mathbf{D}+ \mathbf{k}) \omega(\mathbf{x})^{-1}.
\end{equation*} 
The parameter $\mathbf{k} \in \mathbb{R}^d$ is called \emph{the quasimomentum}. Let $E_l (\mathbf{k})$, $l \in \mathbb{N}$, be consecutive (counted with multiplicities) eigenvalues of the operator $\mathcal{A} (\mathbf{k})$, and let $\varphi_l (\cdot, \mathbf{k})$, $l \in \mathbb{N}$, be the corresponding normalized eigenfunctions:
\begin{equation}
	\label{eigenpairs_eq}
	\omega(\mathbf{x})^{-1} (\mathbf{D}+ \mathbf{k})^* g(\mathbf{x}) (\mathbf{D}+ \mathbf{k}) \omega(\mathbf{x})^{-1} \varphi_l (\mathbf{x}, \mathbf{k}) = E_l (\mathbf{k}) \varphi_l (\mathbf{x}, \mathbf{k}), \qquad l \in \mathbb{N}.
\end{equation} 
The functions $E_l (\mathbf{k})$ are called \emph{band functions}; they are $\widetilde{\Gamma}$-periodic. Next, $\varphi_l (\mathbf{x}, \mathbf{k})$ are $\Gamma$-periodic in $\mathbf{x}$, and the functions $e^{i\left\langle \mathbf{k},\mathbf{x} \right\rangle } \varphi_l (\mathbf{x}, \mathbf{k})$ can be chosen to be $\widetilde{\Gamma}$-periodic in $\mathbf{k}$.

\begin{remark}
	\label{phi_l_bound_remark}
	Multiplying~\eqref{eigenpairs_eq} by $\omega(\mathbf{x})$ from the left  and putting $\phi_l(\mathbf{x}, \mathbf{k}) \coloneqq \omega(\mathbf{x})^{-1} \varphi_l (\mathbf{x}, \mathbf{k})$, we arrive at the following equation for $\phi_l(\mathbf{x}, \mathbf{k})$\textup{:}
	\begin{equation}
		\label{eigenpairs_eq_1}
		(\mathbf{D}+ \mathbf{k})^* g(\mathbf{x}) (\mathbf{D}+ \mathbf{k}) \phi_l (\mathbf{x}, \mathbf{k}) - E_l (\mathbf{k}) \omega(\mathbf{x})^{2} \phi_l (\mathbf{x}, \mathbf{k}) = 0, \qquad l \in \mathbb{N}.
	\end{equation}
	Separating the real and imaginary parts in~\eqref{eigenpairs_eq_1}, we obtain a system of two equations with real-valued coefficients and identical principal parts. In~\textup{\cite[Ch.~VII, \S\,3, Theorem~3.1]{LaU}}, it was proved that solutions of such systems with Dirichlet conditions belong to the H\"{o}lder class as functions of $\mathbf{x}$. However, the proof carries over to the case of periodic boundary conditions without significant changes. This together with $\omega \in L_\infty$ yields $\varphi_l \in L_\infty$, $l \in \mathbb{N}$. See also~\textup{\cite[\S\,4, Sec.~9]{B1997}} and~\textup{\cite[\S\,1, Sec.~1]{BSu2004}}.
\end{remark}

Initially, the Gelfand transformation $\mathscr{G}$ is defined on functions of the Schwartz class \mbox{$v \in \mathcal{S}(\mathbb{R}^d)$} by the formula
\begin{equation*}
	\tilde{v} (\mathbf{x}, \mathbf{k}) = (\mathscr{G} \- v) (\mathbf{x}, \mathbf{k}) = | \widetilde{\Omega} |^{-1/2} \sum_{\mathbf{a} \in \Gamma} e^{- i \left< \mathbf{k}, \mathbf{x} + \mathbf{a} \right> } v ( \mathbf{x} + \mathbf{a}), \qquad \mathbf{x} \in \mathbb{R}^d, \quad \mathbf{k} \in \mathbb{R}^d.
\end{equation*}
The function $\tilde{v} (\mathbf{x}, \mathbf{k})$ is $\Gamma$\nobreakdash-periodic in $\mathbf{x}$ and $\widetilde{\Gamma}$\nobreakdash-quasiperiodic in $\mathbf{k}$ (i.e. the function $e^{i\left\langle \mathbf{x}, \mathbf{k} \right\rangle } \tilde{v} (\mathbf{x}, \mathbf{k})$ is $\widetilde{\Gamma}$\nobreakdash-periodic in $\mathbf{k}$). So, it suffices to consider $\tilde{v} (\mathbf{x}, \mathbf{k})$ for $\mathbf{x} \in \Omega$ and $\mathbf{k} \in \mathbb{T}^d$, where $\mathbb{T}^d$ is the torus $\mathbb{R}^d/\widetilde{\Gamma}$ with the induced $\mathbb{R}^d$-metric. Points of the torus $\mathbf{k} \in \mathbb{T}^d$ can be realized, for example, as points in $\widetilde{\Omega}$. The inverse transform is given by
\begin{equation}
	\label{Gelfand_inverse}
	v(\mathbf{x}) = (\mathscr{G}^{-1} \tilde{v})(\mathbf{x}) = | \widetilde{\Omega} |^{-1/2}  \int_{\mathbb{T}^d} \tilde{v} (\mathbf{x}, \mathbf{k}) e^{ i \left< \mathbf{x}, \mathbf{k} \right> } d  \mathbf{k}, \qquad \mathbf{x} \in \mathbb{R}^d.
\end{equation} 
Since $\int_{\mathbb{T}^d} \int_{\Omega} | \tilde{v} (\mathbf{x}, \mathbf{k}) |^2 d  \mathbf{x} \, d  \mathbf{k} = \int_{\mathbb{R}^d} | v ( \mathbf{x} ) |^2 d \mathbf{x}$, the transformation $\mathscr{G}$ extends by continuity up to a unitary mapping:
\begin{equation*}
	\mathscr{G} \colon L_2 (\mathbb{R}^d) \to \int_{\mathbb{T}^d} \oplus  L_2 (\Omega) \; d \mathbf{k} \eqqcolon \mathcal{K}.
\end{equation*}
The relation $v \in H^1 (\mathbb{R}^d)$ is equivalent to the fact that $\tilde{v} (\cdot, \mathbf{k}) \in \widetilde H^1 (\Omega)$ for a.e. $\mathbf{k} \in \mathbb{T}^d$ and 
\begin{equation*}
	\int_{\mathbb{T}^d} \int_{\Omega} \left( | (\mathbf{D} + \mathbf{k}) \tilde{v} (\mathbf{x}, \mathbf{k}) |^2 + | \tilde{v} (\mathbf{x}, \mathbf{k}) |^2 \right) \, d \mathbf{x} \,  d \mathbf{k} < \infty.
\end{equation*}

Under the Gelfand transformation $\mathscr{G}$, the operator of multiplication by a bounded $\Gamma$-periodic function in $L_2 (\mathbb{R}^d)$ turns into multiplication by the same function on the fibers of the direct integral $\mathcal{K}$. The operator $\mathbf{D}$ applied to $v \in H^1 (\mathbb{R}^d)$ turns into the operator  $\mathbf{D} + \mathbf{k}$ applied to $\tilde{v} (\cdot, \mathbf{k}) \in \widetilde H^1 (\Omega)$.

Under the Gelfand transformation $\mathscr{G}$ the operator $\mathcal{A}$ expands in the direct integral of the operators $\mathcal{A} (\mathbf{k})$:
\begin{equation}
	\label{Gelfand_A_decompose}
	\mathscr{G} \mathcal{A}  \mathscr{G}^{-1} = \int_{\mathbb{T}^d} \oplus \mathcal{A} (\mathbf{k}) \, d \mathbf{k}.
\end{equation}
This means the following. If $v \in \mathcal{H}^1 (\mathbb{R}^d)$, then 
\begin{gather}
	\label{Gelf_1}
	\tilde{v}(\cdot, \mathbf{k}) \in \widetilde{\mathcal{H}}^1 (\Omega) \quad \text{for a.e.} \;\; \mathbf{k} \in \mathbb{T}^d, \\
	\label{Gelf_2}
	\mathfrak{a}[v, v] = \int_{\mathbb{T}^d} \mathfrak{a}(\mathbf{k}) [\tilde{v}(\cdot, \mathbf{k}), \tilde{v}(\cdot, \mathbf{k})] \, d \mathbf{k}.
\end{gather}
Conversely, if $\tilde{v} \in \mathcal{K}$ satisfies~\eqref{Gelf_1} and the integral in~\eqref{Gelf_2} is finite, then $v \in \mathcal{H}^1 (\mathbb{R}^d)$ and~\eqref{Gelf_2} is valid. From~\eqref{Gelfand_A_decompose} it follows that the spectrum of $\mathcal{A}$ is the union of segments (spectral bands) $\Ran E_j$, $j \in \mathbb{N}$.

Introduce the operator $P_0$ acting as averaging over the cell $\Omega$: 
\begin{equation*}
	P_0 u = |\Omega|^{-1} \int_{\Omega} u(\mathbf{x}) \, d \mathbf{x}, \qquad u \in L_2(\Omega).
\end{equation*}
The operator $P_0$ is the orthogonal projection of $L_2(\Omega)$ onto the subspace of constants
\begin{equation*}
	\mathfrak{N}_0 = \{ u \in L_2(\Omega) \colon u = c \in \mathbb{C} \}.
\end{equation*}
The following relation is valid (see, e.g.,~\cite[\S\,6, Sec.~6.1]{BSu2005}):
\begin{equation}
	\label{Gelfand_Fourier}
	([P_0] \mathscr{G} u) (\mathbf{k}) = |\Omega|^{-1/2} (\Phi u) (\mathbf{k}), \qquad u \in L_2(\mathbb{R}^d), \quad \mathbf{k} \in \mathbb{R}^d.
\end{equation}
Here $[P_0]$ is the projection in $\mathcal{K}$ that acts on fibers as the operator $P_0$. Conversely, if  $\operatorname{supp} c \subset \mathrm{B}_{\varkappa}(\mathbf{k}')$ with some $\mathbf{k}' \in \mathbb{R}^d$ and sufficiently small $\varkappa$, and $c(\mathbf{k}) \in \mathfrak{N}_0$, $\mathbf{k} \in \mathrm{B}_{\varkappa}(\mathbf{k}')$, then from~\eqref{Gelfand_inverse} and the relation $|\Omega||\widetilde{\Omega}| = (2 \pi)^d$ it follows that
\begin{equation}
	\label{Gelfand_Fourier_inv}
	(\mathscr{G}^{-1} c) (\mathbf{x}) = |\Omega|^{1/2} (\Phi^* c) (\mathbf{x}).
\end{equation}
In~\eqref{Gelfand_Fourier_inv}, the points $\mathbf{k} \in \mathbb{T}^d$ are realized as points from a set $\widetilde{\Omega}_{\mathbf{k}'}$ such that $\mathrm{B}_{\varkappa}(\mathbf{k}') \subset \widetilde{\Omega}_{\mathbf{k}'}$.

Let us fix some point $\mathbf{k}^\circ \in \mathbb{T}^d$ and a number $s \in \mathbb{N}$. Put $\lambda_0 \coloneqq E_s(\mathbf{k}^\circ)$. Let $n$ be the multiplicity of the eigenvalue $\lambda_0$ of the operator $\mathcal{A}(\mathbf{k}^\circ)$, and let $d_0$ be the distance from the point $\lambda_0$ to the rest of the spectrum of  $\mathcal{A}(\mathbf{k}^\circ)$. By the continuity of the band functions, we can choose $\varkappa > 0$ such that for $|\delta \mathbf{k}| \le \varkappa$, $\delta \mathbf{k} \coloneqq \mathbf{k} - \mathbf{k}^\circ$, there are exactly $n$ eigenvalues (counted with multiplicities) of the operator $\mathcal{A}(\mathbf{k})$ on the segment $[\lambda_0 - d_0/3, \lambda_0 + d_0/3]$, and 
\begin{equation*}
	\bigl([\lambda_0 - 2d_0/3, \lambda_0 - d_0/3] \cup [\lambda_0 + d_0/3, \lambda_0 + 2d_0/3]\bigr) \cap \spec \mathcal{A}(\mathbf{k}) = \varnothing.
\end{equation*}
Introduce the notation $\mathfrak{N} \coloneqq \Ker (\mathcal{A}(\mathbf{k}^\circ) - \lambda_0 I)$. Let $P$ be the orthogonal projection of $L_2(\Omega)$ onto $\mathfrak{N}$; by $F(\mathbf{k})$ we denote the spectral projection of the operator $\mathcal{A}(\mathbf{k})$ corresponding to the segment $[\lambda_0 - d_0/3, \lambda_0 + d_0/3]$.

\section{Threshold approximations}
\label{ch3}
\subsection{Approximations for $F(\mathbf{k})$ and $\mathcal{A}(\mathbf{k}) F(\mathbf{k})$}
In this section, we want to find approximations for the operators $F(\mathbf{k})$ and $\mathcal{A}(\mathbf{k}) F(\mathbf{k})$ for $|\delta \mathbf{k}| \le \varkappa$. For this, we shall integrate the difference of the resolvents for $\mathcal{A}(\mathbf{k})$ and $\mathcal{A}(\mathbf{k}^\circ)$ along an appropriate contour (see, e.g.,~\cite[Ch.~1, Sec.~1.7, \S\S\,2, 3]{BSu2003}, \cite[\S\,4, Sec.~4.2, the third method]{PiSlSuZh2022}). Here we apply the method of~\cite{PiSlSuZh2022}. However, there is the complication that the (standard) second resolvent identity is not applicable, because, in general, the difference $\mathcal{A}(\mathbf{k}) - \mathcal{A}(\mathbf{k}^\circ)$ makes no sense. In order to overcome this difficulty, we use the following lemma. (Here and throughout this section we drop the indices in the inner product and the norm in $L_2(\Omega)$.)
\begin{lemma}
	We have
	\begin{equation}
		\label{resolv_ident_forms}
		\begin{multlined}[c][0.8\textwidth]
			\bigl(\bigl((\mathcal{A}(\mathbf{k})- \zeta I)^{-1} - (\mathcal{A}(\mathbf{k}^\circ)- \zeta I)^{-1} \bigr) \eta, \vartheta\bigr) \\ = -\bigl(\mathfrak{a}(\mathbf{k})- \mathfrak{a}(\mathbf{k}^\circ)\bigr)[(\mathcal{A}(\mathbf{k})- \zeta I)^{-1} \eta, (\mathcal{A}(\mathbf{k}^\circ)- \zeta^* I)^{-1} \vartheta], \\ \eta, \vartheta \in L_2(\Omega), \; \zeta \in \rho(\mathcal{A}(\mathbf{k})) \cap \rho(\mathcal{A}(\mathbf{k}^\circ)).
		\end{multlined}
	\end{equation}
\end{lemma}
\begin{proof}
	Consider the form $(\mathfrak{a}(\mathbf{k})- \mathfrak{a}(\mathbf{k}^\circ))[u, v]$ on the elements $u = (\mathcal{A}(\mathbf{k})- \zeta I)^{-1} \eta$ and $v = (\mathcal{A}(\mathbf{k}^\circ) - \zeta^* I)^{-1} \vartheta$, where $\eta, \vartheta \in L_2(\Omega)$, $\zeta \in \rho(\mathcal{A}(\mathbf{k})) \cap \rho(\mathcal{A}(\mathbf{k}^\circ))$.  Obviously, $(\mathcal{A}(\mathbf{k})- \zeta I)^{-1} \eta \in \Dom\mathcal{A}(\mathbf{k})$ and $\mathcal{A}(\mathbf{k}) (\mathcal{A}(\mathbf{k})- \zeta I)^{-1} = I + \zeta (\mathcal{A}(\mathbf{k})- \zeta I)^{-1}$, whence
	\begin{multline*}
		(\mathfrak{a}(\mathbf{k})- \mathfrak{a}(\mathbf{k}^\circ))[u, v] = (\mathcal{A}(\mathbf{k})u, v) - (u, \mathcal{A}(\mathbf{k}^\circ)v) = (\eta, (\mathcal{A}(\mathbf{k}^\circ)- \zeta^* I)^{-1} \vartheta) - ((\mathcal{A}(\mathbf{k})- \zeta I)^{-1} \eta, \vartheta) \\ +
		(\zeta (\mathcal{A}(\mathbf{k})- \zeta I)^{-1} \eta, (\mathcal{A}(\mathbf{k}^\circ)- \zeta^* I)^{-1} \vartheta) - ((\mathcal{A}(\mathbf{k})- \zeta I)^{-1} \eta,\zeta^* (\mathcal{A}(\mathbf{k}^\circ)- \zeta^* I)^{-1} \vartheta).
	\end{multline*}
	The last two terms cancel out, which yields~\eqref{resolv_ident_forms}.
\end{proof}
Denote
\begin{equation*}
	R(\mathbf{k}, \zeta) \coloneqq (\mathcal{A}(\mathbf{k})- \zeta I)^{-1}, \qquad R_0(\zeta) \coloneqq R(\mathbf{k}^\circ, \zeta) = (\mathcal{A}(\mathbf{k}^\circ)- \zeta I)^{-1}.
\end{equation*}
Let $\gamma$ be a contour on the complex plane that is equidistant to the interval $[\lambda_0 - d_0/3, \lambda_0 + d_0/3]$ and passes through the point $\lambda_0 + d_0/2$. Its length is equal to
\begin{equation*}
	l_\gamma = \frac{\pi + 4}{3} d_0.
\end{equation*}
The resolvent on this contour satisfies the estimates
\begin{equation}
	\label{R_R0_est}
	\| R(\mathbf{k}, \zeta) \| \le 6 d_0^{-1}, \qquad \| R_0(\zeta) \| \le 6 d_0^{-1}, \qquad |\delta \mathbf{k}| \le \varkappa, \quad \zeta \in \gamma.
\end{equation}

Passing from forms to operators, we rewrite identity~\eqref{resolv_ident_forms} as
\begin{equation}
	\label{resolv_ident}
	R(\mathbf{k}, \zeta) =  R_0(\zeta) - \mathcal{Y}_0(\delta \mathbf{k}, \zeta^*)^* \mathcal{X}(\mathbf{k}, \zeta) 
	- \mathcal{X}_0(\zeta^*)^* \mathcal{Y}(\delta \mathbf{k}, \mathbf{k}, \zeta) 
	- \mathcal{Y}_0(\delta \mathbf{k}, \zeta^*)^* \mathcal{Y}(\delta \mathbf{k}, \mathbf{k}, \zeta),
\end{equation}
where
\begin{equation}
	\label{X_X0_Y_Y0}
	\begin{alignedat}{4}
		&\mathcal{X}(\mathbf{k}, \zeta) && =  g^{1/2} (\mathbf{D} + \mathbf{k}^\circ) \omega^{-1} R(\mathbf{k}, \zeta), \qquad 
		&&\mathcal{Y}(\delta \mathbf{k}, \mathbf{k}, \zeta) && =  g^{1/2} (\delta \mathbf{k}) \omega^{-1} R(\mathbf{k}, \zeta), 
		\\
		&\mathcal{X}_0( \zeta) && =  g^{1/2} (\mathbf{D} + \mathbf{k}^\circ) \omega^{-1} R_0(\zeta), \qquad 
		&&\mathcal{Y}_0(\delta \mathbf{k}, \zeta) && =  g^{1/2} (\delta \mathbf{k}) \omega^{-1} R_0(\zeta).
	\end{alignedat}
\end{equation}
Let us estimate the norms of the operators $\mathcal{X}(\mathbf{k}, \zeta)$, $\mathcal{X}_0(\zeta)$, $\mathcal{Y}(\delta \mathbf{k}, \mathbf{k}, \zeta)$ and $\mathcal{Y}_0(\delta \mathbf{k}, \zeta)$ for \mbox{$|\delta \mathbf{k}| \le \varkappa$}, $\zeta \in \gamma$. Clearly,
\begin{align}
	\label{Y_Y0_est}
	&\| \mathcal{Y}(\delta \mathbf{k}, \mathbf{k}, \zeta) \| \le C_1 |\delta \mathbf{k}|, \qquad \| \mathcal{Y}_0(\delta \mathbf{k}, \zeta) \| \le C_1 |\delta \mathbf{k}|, \qquad |\delta \mathbf{k}| \le \varkappa, \quad \zeta \in \gamma; 
	\\ 
	\notag
	&C_1 \coloneqq 6 \|g\|_{L_\infty}^{1/2} \| \omega^{-1} \|_{L_\infty} d_0^{-1}.  
\end{align}
Next, using the identity $\mathcal{A}(\mathbf{k}) R(\mathbf{k}, \zeta) = I + \zeta R(\mathbf{k}, \zeta)$ and taking~\eqref{R_R0_est} into account, we get
\begin{align}
	\label{AR_est}
		&\| \mathcal{A}(\mathbf{k}) R(\mathbf{k}, \zeta) \| \le 1 + (\lambda_0 + d_0/2)(6 d_0^{-1}) = 4 + 6\lambda_0 d_0^{-1}, 
		\qquad | \delta \mathbf{k} | \le \varkappa, \quad  \zeta \in \gamma;
	\\
	\notag
	&\begin{multlined}[c][0.8\textwidth]
		\| \mathcal{A}(\mathbf{k})^{1/2} R(\mathbf{k}, \zeta) u\|^2 = (\mathcal{A}(\mathbf{k}) R(\mathbf{k}, \zeta) u, R(\mathbf{k}, \zeta) u) \le (24 d_0^{-1} +  36 \lambda_0 d_0^{-2}) \|u\|^2, 
		\\
		u \in L_2(\Omega), \; | \delta \mathbf{k} | \le \varkappa, \;  \zeta \in \gamma;
	\end{multlined}
\end{align}
whence
\begin{align}
	\label{X_est}
	&\begin{multlined}[c][0.9\textwidth]
		\| \mathcal{X}(\mathbf{k}, \zeta)\| \le \| g^{1/2} (\mathbf{D}+\mathbf{k}) \omega^{-1} R(\mathbf{k}, \zeta)\| + \| g^{1/2} (\delta \mathbf{k}) \omega^{-1} R(\mathbf{k}, \zeta)\|
		\\
		= \|\mathcal{A}(\mathbf{k})^{1/2} R(\mathbf{k}, \zeta)\| + \|g^{1/2} (\delta \mathbf{k})\omega^{-1} R(\mathbf{k}, \zeta)\| \\ \le  (24 d_0^{-1} +  36 \lambda_0 d_0^{-2})^{1/2} + 6 \| g\|_{L_\infty}^{1/2} \| \omega^{-1} \|_{L_\infty} \varkappa d_0^{-1} \eqqcolon \check{C}_2, \qquad
		| \delta \mathbf{k} | \le \varkappa, \; \zeta \in \gamma;
	\end{multlined}
	\\
	\label{X0_est}
	&\| \mathcal{X}_0(\zeta)\| \le (24 d_0^{-1} +  36 \lambda_0 d_0^{-2})^{1/2} \eqqcolon C_2, \qquad | \delta \mathbf{k} | \le \varkappa, \; \zeta \in \gamma.
\end{align}

Now, iterating, we apply identity~\eqref{resolv_ident} for the resolvent $R(\mathbf{k}, \zeta)$, contained in the terms $\mathcal{Y}_0(\delta \mathbf{k}, \zeta^*)^* \mathcal{X}(\mathbf{k}, \zeta)$ and
$\mathcal{X}_0(\zeta^*)^* \mathcal{Y}(\delta \mathbf{k}, \mathbf{k}, \zeta)$ in~\eqref{resolv_ident}. Thus, the terms of order $|\delta \mathbf{k}|$ will not contain $R(\mathbf{k}, \zeta)$:
\begin{equation}
	\label{resolv_iterat_1}
	R(\mathbf{k}, \zeta) 
	= R_0(\zeta) 
	- \mathcal{Y}_0(\delta \mathbf{k}, \zeta^*)^* \mathcal{X}_0( \zeta) 
	- \mathcal{X}_0( \zeta^*)^* \mathcal{Y}_0(\delta \mathbf{k}, \zeta)
	+ \mathscr{R}_1(\delta \mathbf{k}, \mathbf{k}, \zeta).
\end{equation}
Here $\mathscr{R}_1(\delta \mathbf{k}, \mathbf{k}, \zeta)$ is defined by the expression
\begin{multline*}
	\mathscr{R}_1(\delta \mathbf{k}, \mathbf{k}, \zeta) =  
	\mathcal{Y}_0(\delta \mathbf{k}, \zeta^*)^* \mathcal{X}_0( \zeta) \check{\mathcal{Y}}(\delta \mathbf{k})^* \mathcal{X}(\mathbf{k}, \zeta) 
	+ \mathcal{Y}_0(\delta \mathbf{k}, \zeta^*)^* \check{\mathcal{X}}_0(\zeta) \mathcal{Y}(\delta \mathbf{k}, \mathbf{k}, \zeta)
	\\
	+ \mathcal{Y}_0(\delta \mathbf{k}, \zeta^*)^* \mathcal{X}_0(\zeta) \check{\mathcal{Y}}(\delta \mathbf{k})^*  \mathcal{Y}(\delta \mathbf{k}, \mathbf{k}, \zeta)
	+ \mathcal{X}_0(\zeta^*)^* \mathcal{Y}_0(\delta \mathbf{k}, \zeta) \check{\mathcal{Y}}(\delta \mathbf{k})^*  \mathcal{X}(\mathbf{k}, \zeta) 
	\\
	+ \mathcal{X}_0(\zeta^*)^* \check{\mathcal{Y}}(\delta \mathbf{k})   \mathcal{X}_0(\zeta^*)^* \mathcal{Y}(\delta \mathbf{k}, \mathbf{k}, \zeta) 
	+ \mathcal{X}_0(\zeta^*)^* \mathcal{Y}_0(\delta \mathbf{k}, \zeta) \check{\mathcal{Y}}(\delta \mathbf{k})^* \mathcal{Y}(\delta \mathbf{k}, \mathbf{k}, \zeta) 
	- \mathcal{Y}_0(\delta \mathbf{k}, \zeta^*)^* \mathcal{Y}(\delta \mathbf{k}, \mathbf{k}, \zeta).
\end{multline*}
We have denoted
\begin{equation}
	\label{X0_check_Y_check}
	\check{\mathcal{X}}_0 (\zeta) \coloneqq g^{1/2} (\mathbf{D}+\mathbf{k}^\circ) \omega^{-1} \bigl(g^{1/2} (\mathbf{D}+\mathbf{k}^\circ) \omega^{-1} R_0(\zeta^*)\bigr)^*, \qquad \check{\mathcal{Y}}(\delta \mathbf{k}) \coloneqq g^{1/2}  (\delta \mathbf{k}) \omega^{-1}.
\end{equation}
Let us estimate the norms of the operators $\check{\mathcal{X}}_0 (\zeta)$ and $\check{\mathcal{Y}}(\delta \mathbf{k})$. Write $\check{\mathcal{X}}_0 (\zeta)$ as
\begin{multline*}
	\check{\mathcal{X}}_0 (\zeta) = g^{1/2} (\mathbf{D}+\mathbf{k}^\circ) \omega^{-1} \bigl(g^{1/2} (\mathbf{D}+\mathbf{k}^\circ) \omega^{-1} R_0(\zeta^*) \mathcal{A}(\mathbf{k}^\circ)^{1/2} \mathcal{A}(\mathbf{k}^\circ)^{-1/2}\bigr)^* 
	\\ = g^{1/2} (\mathbf{D}+\mathbf{k}^\circ) \omega^{-1} \mathcal{A}(\mathbf{k}^\circ)^{-1/2} \bigl(g^{1/2} (\mathbf{D}+\mathbf{k}^\circ) \omega^{-1} R_0(\zeta^*) \mathcal{A}(\mathbf{k}^\circ)^{1/2} \bigr)^*
	\\ = g^{1/2} (\mathbf{D}+\mathbf{k}^\circ) \omega^{-1} \mathcal{A}(\mathbf{k}^\circ)^{-1/2} \bigl(g^{1/2} (\mathbf{D}+\mathbf{k}^\circ) \omega^{-1}  \mathcal{A}(\mathbf{k}^\circ)^{-1/2} \mathcal{A}(\mathbf{k}^\circ) R_0(\zeta^*) \bigr)^*.
\end{multline*}
We have
\begin{equation*}
	\| \check{\mathcal{X}}_0 (\zeta) \| \le \| g^{1/2} (\mathbf{D}+\mathbf{k}^\circ) \omega^{-1} \mathcal{A}(\mathbf{k}^\circ)^{-1/2} \|^2 \| \mathcal{A}(\mathbf{k}^\circ) R_0(\zeta^*)\|.
\end{equation*}
By~\eqref{AR_est} and the identity $\| g^{1/2} (\mathbf{D}+\mathbf{k}^\circ) \omega^{-1} \mathcal{A}(\mathbf{k}^\circ)^{-1/2} \| = 1$,
\begin{equation}
	\label{X0_check_est}
		\| \check{\mathcal{X}}_0 (\zeta) \| \le 4 + 6 \lambda_0 d_0^{-1} \eqqcolon C_3,
		\qquad 
		| \delta \mathbf{k} | \le \varkappa, \; \zeta \in \gamma.
\end{equation}
Next, obviously, 
\begin{equation}
	\label{Y_check_est}
	\| \check{\mathcal{Y}}(\delta \mathbf{k}) \| \le C_4 |\delta \mathbf{k}|, \quad C_4 \coloneqq \|g\|_{L_\infty}^{1/2} \| \omega^{-1} \|_{L_\infty}, \qquad | \delta \mathbf{k} | \le \varkappa, \quad \zeta \in \gamma.
\end{equation}

In order to get rid of the resolvent $R(\mathbf{k}, \zeta)$ in the terms of order $|\delta \mathbf{k}|^2$, we apply identity~\eqref{resolv_ident} once again:
\begin{equation}
	\label{resolv_iterat_2}
	\begin{multlined}
		R(\mathbf{k}, \zeta) = 
		R_0(\zeta) 
		- \mathcal{Y}_0(\delta \mathbf{k}, \zeta^*)^* \mathcal{X}_0( \zeta) 
		- \mathcal{X}_0( \zeta^*)^* \mathcal{Y}_0(\delta \mathbf{k}, \zeta)
		- \mathcal{Y}_0(\delta \mathbf{k}, \zeta^*)^* \mathcal{Y}_0(\delta \mathbf{k}, \zeta) 
		\\
		+ \mathcal{Y}_0(\delta \mathbf{k}, \zeta^*)^* \mathcal{X}_0( \zeta) \check{\mathcal{Y}}(\delta \mathbf{k})^* \mathcal{X}_0( \zeta) 
		+ \mathcal{Y}_0(\delta \mathbf{k}, \zeta^*)^* \check{\mathcal{X}}_0( \zeta) \mathcal{Y}_0(\delta \mathbf{k}, \zeta)
		\\ 
		+ \mathcal{X}_0( \zeta^*)^* \mathcal{Y}_0(\delta \mathbf{k}, \zeta) \check{\mathcal{Y}}(\delta \mathbf{k})^* \mathcal{X}_0( \zeta)
		+ \mathcal{X}_0( \zeta^*)^* \check{\mathcal{Y}}(\delta \mathbf{k}) \mathcal{X}_0( \zeta^*)^* \mathcal{Y}_0(\delta \mathbf{k}, \zeta)
		+ \mathscr{R}_2(\delta \mathbf{k}, \mathbf{k}, \zeta),
	\end{multlined}
\end{equation}
where
\begin{multline*}
	\mathscr{R}_2(\delta \mathbf{k}, \mathbf{k}, \zeta) = \mathcal{Y}_0(\delta \mathbf{k}, \zeta^*)^* \mathcal{X}_0(\zeta) \check{\mathcal{Y}}(\delta \mathbf{k})^*  \mathcal{Y}(\delta \mathbf{k}, \mathbf{k}, \zeta) + \mathcal{X}_0(\zeta^*)^* \mathcal{Y}_0(\delta \mathbf{k}, \zeta) \check{\mathcal{Y}}(\delta \mathbf{k})^* \mathcal{Y}(\delta \mathbf{k}, \mathbf{k}, \zeta)
	\\
	- \mathcal{Y}_0(\delta \mathbf{k}, \zeta^*)^* \mathcal{X}_0( \zeta) \check{\mathcal{Y}}(\delta \mathbf{k})^* \mathcal{X}_0( \zeta) \check{\mathcal{Y}}(\delta \mathbf{k})^* \mathcal{X}(\mathbf{k}, \zeta)
	- \mathcal{Y}_0(\delta \mathbf{k}, \zeta^*)^* \mathcal{X}_0( \zeta) \check{\mathcal{Y}}(\delta \mathbf{k})^* \check{\mathcal{X}}_0( \zeta) \mathcal{Y}(\delta \mathbf{k}, \mathbf{k}, \zeta) 
	\\
	- \mathcal{Y}_0(\delta \mathbf{k}, \zeta^*)^* \mathcal{X}_0( \zeta) \check{\mathcal{Y}}(\delta \mathbf{k})^* \mathcal{X}_0( \zeta) \check{\mathcal{Y}}(\delta \mathbf{k})^* \mathcal{Y}(\delta \mathbf{k}, \mathbf{k}, \zeta) 
	-  \mathcal{Y}_0(\delta \mathbf{k}, \zeta^*)^* \check{\mathcal{X}}_0( \zeta) \check{\mathcal{Y}}(\delta \mathbf{k}) \mathcal{Y}_0(\delta \mathbf{k}, \zeta^*)^* \mathcal{X}(\mathbf{k}, \zeta)
	\\
	-  \mathcal{Y}_0(\delta \mathbf{k}, \zeta^*)^* \check{\mathcal{X}}_0( \zeta) \check{\mathcal{Y}}(\delta \mathbf{k}) \mathcal{X}_0( \zeta^*)^* \mathcal{Y}(\delta \mathbf{k}, \mathbf{k}, \zeta) 
	-  \mathcal{Y}_0(\delta \mathbf{k}, \zeta^*)^* \check{\mathcal{X}}_0( \zeta) \check{\mathcal{Y}}(\delta \mathbf{k}) \mathcal{Y}_0(\delta \mathbf{k}, \zeta^*)^* \mathcal{Y}(\delta \mathbf{k}, \mathbf{k}, \zeta)
	\\
	- \mathcal{X}_0( \zeta^*)^*  \mathcal{Y}_0(\delta \mathbf{k}, \zeta) \check{\mathcal{Y}}(\delta \mathbf{k})^* \mathcal{X}_0( \zeta) \check{\mathcal{Y}}(\delta \mathbf{k}) \mathcal{X}(\mathbf{k}, \zeta) 
	- \mathcal{X}_0( \zeta^*)^*  \mathcal{Y}_0(\delta \mathbf{k}, \zeta) \check{\mathcal{Y}}(\delta \mathbf{k})^* \check{\mathcal{X}}_0( \zeta) \mathcal{Y}(\delta \mathbf{k}, \mathbf{k}, \zeta) 
	\\
	- \mathcal{X}_0( \zeta^*)^*  \mathcal{Y}_0(\delta \mathbf{k}, \zeta) \check{\mathcal{Y}}(\delta \mathbf{k})^* \mathcal{X}_0( \zeta) \check{\mathcal{Y}}(\delta \mathbf{k})^* \mathcal{Y}(\delta \mathbf{k}, \mathbf{k}, \zeta)  
	- \mathcal{X}_0( \zeta^*)^* \check{\mathcal{Y}}(\delta \mathbf{k}) \mathcal{X}_0( \zeta^*)^*  \check{\mathcal{Y}}(\delta \mathbf{k}) \mathcal{Y}_0(\delta \mathbf{k}, \zeta^*)^* \mathcal{X}(\mathbf{k}, \zeta)
	\\	
	- \mathcal{X}_0( \zeta^*)^* \check{\mathcal{Y}}(\delta \mathbf{k}) \mathcal{X}_0( \zeta^*)^* \check{\mathcal{Y}}(\delta \mathbf{k}) \mathcal{X}_0( \zeta^*)^* \mathcal{Y}(\delta \mathbf{k}, \mathbf{k}, \zeta) 
	- \mathcal{X}_0( \zeta^*)^* \check{\mathcal{Y}}(\delta \mathbf{k}) \mathcal{X}_0( \zeta^*)^* \check{\mathcal{Y}}(\delta \mathbf{k}) \mathcal{Y}_0(\delta \mathbf{k}, \zeta^*)^* \mathcal{Y}(\delta \mathbf{k}, \mathbf{k}, \zeta) 
	\\
	+ \mathcal{Y}_0(\delta \mathbf{k}, \zeta^*)^* \check{\mathcal{Y}}(\delta \mathbf{k}) \mathcal{Y}_0(\delta \mathbf{k}, \zeta^*)^* \mathcal{X}(\mathbf{k}, \zeta)
	+ \mathcal{Y}_0(\delta \mathbf{k}, \zeta^*)^* \check{\mathcal{Y}}(\delta \mathbf{k})  \mathcal{X}_0( \zeta^*)^* \mathcal{Y}(\delta \mathbf{k}, \mathbf{k}, \zeta) 
	\\
	+ \mathcal{Y}_0(\delta \mathbf{k}, \zeta^*)^* \check{\mathcal{Y}}(\delta \mathbf{k}) \mathcal{Y}_0(\delta \mathbf{k}, \zeta^*)^* \mathcal{Y}(\delta \mathbf{k}, \mathbf{k}, \zeta).
\end{multline*}

The operators $\mathscr{R}_1(\delta \mathbf{k}, \mathbf{k}, \zeta)$ and $\mathscr{R}_2(\delta \mathbf{k}, \mathbf{k}, \zeta)$ satisfy the estimates
\begin{gather}
	\label{scrR1_scrR2_est}
	\| \mathscr{R}_1(\delta \mathbf{k}, \mathbf{k}, \zeta) \| \le C_5 | \delta \mathbf{k} |^2, \qquad \| \mathscr{R}_2(\delta \mathbf{k}, \mathbf{k}, \zeta) \| \le C_6 | \delta \mathbf{k} |^3, \qquad | \delta \mathbf{k} | \le \varkappa, \; \zeta \in \gamma;
	\\
	\notag
	C_5 = 2 C_1 C_2 \check{C}_2 C_4 + C_1^2 C_3 + 2 C_1^2 C_2 C_4 \varkappa + C_1 C_2^2 C_4 + C_1^2, 
	\\
	\notag
	\begin{multlined}
		C_6 
		= 3 C_1^2 C_2 C_4 
		+ 3 C_1 C_2^2 \check{C}_2 C_4^2 
		+ 3 C_1^2 C_2 C_3 C_4 
		+ 3 C_1^2 C_2^2 C_4^2 \varkappa 
		\\
		+ C_1^2 \check{C}_2 C_3 C_4 
		+ C_1^3 C_3 C_4 \varkappa
		+ C_1 C_2^3 C_4^2 
		+ C_1^2 \check{C}_2 C_4   
		+ C_1^3 C_4 \varkappa.
	\end{multlined}
\end{gather}
In this section, our goal is to find approximations for $F(\mathbf{k})$ and $\mathcal{A}(\mathbf{k}) F(\mathbf{k})$. Let us start with the operator $F(\mathbf{k})$. By virtue of Riesz--Dunford operator calculus,
\begin{equation}
	\label{F(k)_integral}
	F(\mathbf{k}) = \frac{-1}{2 \pi i} \ointctrclockwise_\gamma (\mathcal{A}(\mathbf{k})- \zeta I)^{-1} d \zeta.
\end{equation}
Substituting~\eqref{resolv_ident} into~\eqref{F(k)_integral} and using relations~\eqref{Y_Y0_est}, \eqref{X_est}, \eqref{X0_est}, and the identity \mbox{$P = \frac{-1}{2 \pi i} \ointctrclockwise_\gamma (\mathcal{A}(\mathbf{k}^\circ)- \zeta I)^{-1} d \zeta$}, we obtain the following result:
\begin{gather}
	\label{F(k)_threshold_1}
	\| F(\mathbf{k}) - P \| \le C_7 |\delta \mathbf{k}|, \qquad |\delta \mathbf{k}| \le \varkappa;
	\\
	\notag
	C_7 = (2 \pi)^{-1} l_\gamma (C_1 C_2 + C_1 \check{C}_2 + C_1^2 \varkappa).
\end{gather}
We also need more precise approximation for the projector $F(\mathbf{k})$. For this, we substitute~\eqref{resolv_iterat_1} into~\eqref{F(k)_integral}:
\begin{gather}
	\label{F(k)_threshold_2}
	F(\mathbf{k}) = P + F_1(\delta \mathbf{k}) + \Phi(\delta \mathbf{k}, \mathbf{k}), \qquad |\delta \mathbf{k}| \le \varkappa;
	\\
	\label{F1_integral}
	F_1(\delta \mathbf{k}) \coloneqq \frac{1}{2 \pi i} \ointctrclockwise_\gamma \bigl(\mathcal{Y}_0(\delta \mathbf{k}, \zeta^*)^* \mathcal{X}_0( \zeta) 	+ \mathcal{X}_0( \zeta^*)^* \mathcal{Y}_0(\delta \mathbf{k}, \zeta) \bigr) d\zeta, 
	\\
	\notag
	\Phi(\delta \mathbf{k}, \mathbf{k}) \coloneqq \frac{-1}{2 \pi i} \ointctrclockwise_\gamma \mathscr{R}_1(\delta \mathbf{k}, \mathbf{k}, \zeta) d\zeta.
\end{gather}
Calculate the integral in the expression for $F_1(\delta \mathbf{k})$. Recall notation~\eqref{X_X0_Y_Y0}, take into account the decomposition of the resolvent
\begin{equation}
	\label{resolv_decomp}
	R_0(\zeta) = R_0(\zeta) P + R_0(\zeta)P^\perp = (\lambda_0 - \zeta)^{-1} P + R_0(\zeta)P^\perp, \qquad \zeta \in \gamma,
\end{equation} 
the holomorphy of the operator-valued function $R^\perp_0(\zeta) \coloneqq R_0(\zeta)P^\perp$ inside the contour $\gamma$, the equality $\ointctrclockwise_\gamma (\lambda_0 - \zeta)^{-2} d \zeta = 0$, and use the fact that integral over $\gamma$ of a holomorphic function inside the contour is equal to zero. Therefore,
\begin{equation}
	\label{F1_F1times}
	\begin{multlined}[c][0.9\textwidth]
		F_1(\delta \mathbf{k}) = \frac{1}{2 \pi i} \ointctrclockwise_\gamma \Bigl( \Bigl(\frac{1}{\lambda_0 - \zeta} P + R^\perp_0(\zeta) \Bigr) \omega^{-1} (\delta \mathbf{k})^* g (\mathbf{D} + \mathbf{k}^\circ) \omega^{-1} \Bigl(\frac{1}{\lambda_0 - \zeta} P + R^\perp_0(\zeta)\Bigr) \\ + \Bigl((\mathbf{D} + \mathbf{k}^\circ) \omega^{-1} \Bigl(\frac{1}{\lambda_0 - \zeta^*} P + R^\perp_0(\zeta^*)\Bigr)\Bigr)^*  g (\delta \mathbf{k}) \omega^{-1} \Bigl(\frac{1}{\lambda_0 - \zeta} P + R^\perp_0(\zeta)\Bigr) \Bigr) \, d \zeta \\
		= F^\times_1 (\delta \mathbf{k}) + F^\times_1(\delta \mathbf{k})^*,
	\end{multlined}
\end{equation}
where the operator $F^\times_1(\delta \mathbf{k})$ takes $\mathfrak{N}^\perp$ into $\mathfrak{N}$ and is defined by the expression
\begin{multline}
	\label{F1times}
	F^\times_1(\delta \mathbf{k}) \\ = \tfrac{1}{2 \pi i} \ointctrclockwise_\gamma \Bigl( \tfrac{1}{\lambda_0 - \zeta} P \omega^{-1} (\delta \mathbf{k})^* g (\mathbf{D} + \mathbf{k}^\circ) \omega^{-1}  R^\perp_0(\zeta) + \Bigl((\mathbf{D} + \mathbf{k}^\circ) \omega^{-1} \tfrac{1}{\lambda_0 - \zeta^*}  P \Bigr)^*  g (\delta \mathbf{k}) \omega^{-1} R^\perp_0(\zeta)\Bigr) \, d \zeta
	\\
	= - P \omega^{-1} (\delta \mathbf{k})^* g (\mathbf{D} + \mathbf{k}^\circ) \omega^{-1} R^\perp_0(\lambda_0) - ((\mathbf{D} + \mathbf{k}^\circ) \omega^{-1}P)^* g (\delta \mathbf{k}) \omega^{-1} R^\perp_0(\lambda_0).
\end{multline}
By~\eqref{scrR1_scrR2_est}, the remainder $\Phi(\delta \mathbf{k}, \mathbf{k})$ satisfies the estimate
\begin{equation}
	\label{Phi_est}
	\| \Phi(\delta \mathbf{k}, \mathbf{k}) \| \le C_8 |\delta \mathbf{k}|^2, \qquad |\delta \mathbf{k}| \le \varkappa; \qquad  C_8 \coloneqq (2 \pi)^{-1} l_\gamma C_5.
\end{equation}
Using integral representation~\eqref{F1_integral} and~\eqref{Y_Y0_est}, \eqref{X0_est}, and the relation $F^\times_1(\delta \mathbf{k}) = P F_1(\delta \mathbf{k}) P^\perp$, we estimate the operator $F^\times_1(\delta \mathbf{k})$ as follows:
\begin{equation}
	\label{F1times_est}
	\|F^\times_1(\delta \mathbf{k})\| \le \pi^{-1} l_\gamma C_1 C_2  |\delta \mathbf{k}|, \qquad |\delta \mathbf{k}| \le \varkappa.
\end{equation}
We also note that
\begin{equation}
	\label{F1timesP=0}
	F^\times_1(\delta \mathbf{k}) P = 0, \qquad P F^\times_1(\delta \mathbf{k})^* = 0, \qquad P^\perp F^\times_1(\delta \mathbf{k})  = 0.
\end{equation}
Moreover, we need to consider the operator $F^\times_1(\delta \mathbf{k}) F(\mathbf{k})$. Applying~\eqref{F(k)_threshold_2}, \eqref{F1_F1times}, \eqref{Phi_est},~\eqref{F1times_est}, the first and the third equalities~\eqref{F1timesP=0}, we obtain that
\begin{align}
	\label{F1times_F}
	&
	\begin{multlined}[c][0.9\textwidth]
		F^\times_1(\delta \mathbf{k}) F(\mathbf{k}) = F^\times_1(\delta \mathbf{k}) (F(\mathbf{k}) - P) F(\mathbf{k})
		\\
		 = F^\times_1(\delta \mathbf{k}) F^\times_1(\delta \mathbf{k})^* F(\mathbf{k}) + F^\times_1(\delta \mathbf{k}) \Phi(\delta \mathbf{k}, \mathbf{k}) F(\mathbf{k});
	\end{multlined}
	\\
	\label{F1times_Phi_F_est}
	&\|F^\times_1(\delta \mathbf{k}) \Phi(\delta \mathbf{k}, \mathbf{k}) F(\mathbf{k})\| \le C_9 |\delta \mathbf{k}|^3, \quad |\delta \mathbf{k}| \le \varkappa; \qquad C_9 \coloneqq 2^{-1} \pi^{-2} l_\gamma^2 C_1 C_2 C_5 .
\end{align}
The operator $F^\times_1(\delta \mathbf{k}) F^\times_1(\delta \mathbf{k})^*$ has the form
\begin{equation}
	\label{F1times_F1times*}
	\begin{multlined}[c][0.85\textwidth]
		F^\times_1(\delta \mathbf{k}) F^\times_1(\delta \mathbf{k})^* 
		\\= 
		P \omega^{-1} (\delta \mathbf{k})^* g (\mathbf{D} + \mathbf{k}^\circ) \omega^{-1} R^\perp_0(\lambda_0) \bigl((\mathbf{D} + \mathbf{k}^\circ) \omega^{-1} R^\perp_0(\lambda_0)\bigr)^* g (\delta \mathbf{k}) \omega^{-1} P 
		\\
		+ P \omega^{-1} (\delta \mathbf{k})^* g (\mathbf{D} + \mathbf{k}^\circ) \omega^{-1} R^\perp_0(\lambda_0)^2 \omega^{-1} (\delta \mathbf{k})^* g (\mathbf{D} + \mathbf{k}^\circ) \omega^{-1} P
		\\
		+ \bigl((\mathbf{D} + \mathbf{k}^\circ) \omega^{-1} P\bigr)^* g (\delta \mathbf{k}) \omega^{-1} R^\perp_0(\lambda_0) \bigl((\mathbf{D} + \mathbf{k}^\circ) \omega^{-1} R^\perp_0(\lambda_0)\bigr)^* g (\delta \mathbf{k}) \omega^{-1} P
		\\
		+ \bigl((\mathbf{D} + \mathbf{k}^\circ) \omega^{-1} P\bigr)^* g (\delta \mathbf{k}) \omega^{-1} R^\perp_0(\lambda_0)^2 \omega^{-1} (\delta \mathbf{k})^* g (\mathbf{D} + \mathbf{k}^\circ) \omega^{-1} P.
	\end{multlined}
\end{equation}

Let us now turn to the threshold approximation for the operator $\mathcal{A}(\mathbf{k}) F(\mathbf{k})$. We have
\begin{equation}
		\label{A(k)F(k)_integral}
		\mathcal{A}(\mathbf{k}) F(\mathbf{k}) = \frac{-1}{2 \pi i} \ointctrclockwise_\gamma \zeta (\mathcal{A}(\mathbf{k})- \zeta I)^{-1} d \zeta.
\end{equation}
By substituting~\eqref{resolv_iterat_2} into~\eqref{A(k)F(k)_integral}, one obtains
\begin{equation}
	\label{AF_threshold}
	\mathcal{A}(\mathbf{k}) F(\mathbf{k}) = \lambda_0 P + G_1(\delta \mathbf{k}) + G_2(\delta \mathbf{k}) + \Xi(\delta \mathbf{k}, \mathbf{k}), \qquad |\delta \mathbf{k}| \le \varkappa,
\end{equation}
where
\begin{align}
	\notag
	&G_1(\delta \mathbf{k}) = \frac{1}{2 \pi i} \ointctrclockwise_\gamma \zeta \bigl(\mathcal{Y}_0(\delta \mathbf{k}, \zeta^*)^* \mathcal{X}_0( \zeta) + \mathcal{X}_0( \zeta^*)^* \mathcal{Y}_0(\delta \mathbf{k}, \zeta) \bigr) d\zeta,
	\\
	\label{G2_integral}
	&G_2(\delta \mathbf{k})  = \frac{-1}{2 \pi i} \ointctrclockwise_\gamma \zeta T(\delta \mathbf{k}, \zeta) d\zeta,		
	\\
	\label{T}
	&\begin{multlined}[c][0.85\textwidth]
		T(\delta \mathbf{k}, \zeta) \coloneqq \mathcal{Y}_0(\delta \mathbf{k}, \zeta^*)^* \mathcal{X}_0( \zeta) \check{\mathcal{Y}}(\delta \mathbf{k})^* \mathcal{X}_0( \zeta) 
		+ \mathcal{Y}_0(\delta \mathbf{k}, \zeta^*)^* \check{\mathcal{X}}_0( \zeta) \mathcal{Y}_0(\delta \mathbf{k}, \zeta) 
		\\
		+ \mathcal{X}_0( \zeta^*)^* \mathcal{Y}_0(\delta \mathbf{k}, \zeta) \check{\mathcal{Y}}(\delta \mathbf{k})^* \mathcal{X}_0( \zeta)
		+ \mathcal{X}_0( \zeta^*)^* \check{\mathcal{Y}}(\delta \mathbf{k}) \mathcal{X}_0( \zeta^*)^* \mathcal{Y}_0(\delta \mathbf{k}, \zeta)
		\\
		- \mathcal{Y}_0(\delta \mathbf{k}, \zeta^*)^* \mathcal{Y}_0(\delta \mathbf{k}, \zeta),
	\end{multlined}
	\\
	\label{Xi_integral}
	&\Xi(\delta \mathbf{k}, \mathbf{k}) = \frac{-1}{2 \pi i} \ointctrclockwise_\gamma \zeta \mathscr{R}_2(\delta \mathbf{k}, \mathbf{k}, \zeta) d\zeta.
\end{align}
Recall the definitions of operators~\eqref{X_X0_Y_Y0}, \eqref{X0_check_Y_check}, decomposition of the resolvent~\eqref{resolv_decomp}, and consider first the representation for $G_1(\delta \mathbf{k})$:
\begin{multline*}
	G_1(\delta \mathbf{k}) = \frac{1}{2 \pi i} \ointctrclockwise_\gamma \zeta \Bigl( \Bigl(\frac{1}{\lambda_0 - \zeta} P + R^\perp_0(\zeta) \Bigr) \omega^{-1} (\delta \mathbf{k})^* g (\mathbf{D} + \mathbf{k}^\circ) \omega^{-1} \Bigl(\frac{1}{\lambda_0 - \zeta} P + R^\perp_0(\zeta)\Bigr) \\ + \Bigl((\mathbf{D} + \mathbf{k}^\circ) \omega^{-1} \Bigl(\frac{1}{\lambda_0 - \zeta^*} P + R^\perp_0(\zeta^*)\Bigr)\Bigr)^*  g (\delta \mathbf{k}) \omega^{-1} \Bigl(\frac{1}{\lambda_0 - \zeta} P + R^\perp_0(\zeta)\Bigr) \Bigr) \, d \zeta.
\end{multline*}
Similarly to~\eqref{F1times}, calculating the integral with the help of the formula for the derivative of the Cauchy integral $f'(z) = \frac{1}{2\pi i} \ointctrclockwise_\gamma f(\zeta) (\zeta - z)^{-2} d \zeta$ (where $f$ is a holomorphic function in the domain restricted by the contour~$\gamma$), we have
\begin{gather}
	\label{G1}
	G_1(\delta \mathbf{k}) = \mathfrak{G}^\circ_1 (\delta \mathbf{k}) + \lambda_0 \bigl(F^\times_1(\delta \mathbf{k}) + F^\times_1(\delta \mathbf{k})^*\bigr),
	\\
	\label{frakG1circ}
	\mathfrak{G}^\circ_1 (\delta \mathbf{k}) \coloneqq P \omega^{-1} (\delta \mathbf{k})^* g (\mathbf{D} + \mathbf{k}^\circ) \omega^{-1} P + ((\mathbf{D} + \mathbf{k}^\circ) \omega^{-1} P)^* g (\delta \mathbf{k}) \omega^{-1} P, \qquad \mathfrak{G}^\circ_1 (\delta \mathbf{k}) \colon \mathfrak{N} \to \mathfrak{N}.
\end{gather}
Recall that $F^\times_1(\delta \mathbf{k})$ was calculated in~\eqref{F1times}.
By virtue of~\eqref{AF_threshold},
\begin{equation}
	\label{G1=0_cond}
		G_1(\delta \mathbf{k}) = 0, \qquad \begin{aligned}
		&\text{if the form } \bigl((\mathcal{A}(\mathbf{k}) F(\mathbf{k})  - \lambda_0 P ) u, u\bigr), \; u \in L_2(\Omega), \; |\delta \mathbf{k}| \le \varkappa,
		\\
		&\text{is sign-definite.}
	\end{aligned}
\end{equation}
This together with~\eqref{F(k)_threshold_1}, \eqref{F1times}, the second equality~\eqref{F1timesP=0}, \eqref{F1times_F}, \eqref{F1times_Phi_F_est}, and~\eqref{G1} yields
\begin{multline*}
	P G_1(\delta \mathbf{k}) F(\mathbf{k}) = P \bigl( \mathfrak{G}^\circ_1 (\delta \mathbf{k}) + \lambda_0 F^\times_1(\delta \mathbf{k}) F^\times_1(\delta \mathbf{k})^* + \lambda_0 F^\times_1(\delta \mathbf{k}) \Phi(\delta \mathbf{k}, \mathbf{k}) \bigr) F(\mathbf{k}) \\ = P \mathfrak{G}^\circ_1 (\delta \mathbf{k}) P + O(|\delta \mathbf{k}|^2) = 0,
\end{multline*}
if the condition in~\eqref{G1=0_cond} is satisfied. Thus,
\begin{equation}
	\label{frakG1=0_cond}
	\mathfrak{G}^\circ_1 (\delta \mathbf{k}) = 0, \qquad \begin{aligned}
		&\text{if the form } \bigl((\mathcal{A}(\mathbf{k}) F(\mathbf{k})  - \lambda_0 P ) u, u\bigr), \; u \in L_2(\Omega), \; |\delta \mathbf{k}| \le \varkappa,
		\\
		&\text{is sign-definite.}
		\end{aligned}
\end{equation}

Let us now turn to expression~\eqref{G2_integral} for $G_2(\delta \mathbf{k})$. We write~\eqref{T} as
\begin{equation}
	\label{T_decomp}
	T(\delta \mathbf{k}, \zeta) = T^\circ(\delta \mathbf{k}, \zeta) + T^{\times}(\delta \mathbf{k}, \zeta) + T^{\times}(\delta \mathbf{k}, \zeta^*)^* + T^\perp(\delta \mathbf{k}, \zeta),
\end{equation}
where
\begin{align*}
	&\begin{aligned}
		T^\circ&(\delta \mathbf{k}, \zeta) = P T(\delta \mathbf{k}, \zeta) P
		\\ &= \frac{1}{\lambda_0 - \zeta} P \omega^{-1} (\delta \mathbf{k})^* g (\mathbf{D} + \mathbf{k}^\circ) \omega^{-1} R_0(\zeta) \omega^{-1} (\delta \mathbf{k})^* g (\mathbf{D} + \mathbf{k}^\circ) \omega^{-1} \frac{1}{\lambda_0 - \zeta} P
		\\
		&+ \frac{1}{\lambda_0 - \zeta} P \omega^{-1} (\delta \mathbf{k})^* g (\mathbf{D} + \mathbf{k}^\circ) \omega^{-1} \bigl(  (\mathbf{D} + \mathbf{k}^\circ) \omega^{-1} R_0(\zeta^*)\bigr)^* g (\delta \mathbf{k}) \omega^{-1} \frac{1}{\lambda_0 - \zeta} P
		\\
		&+ \Bigl( (\mathbf{D} + \mathbf{k}^\circ) \omega^{-1} \frac{1}{\lambda_0 - \zeta^*} P \Bigr)^* g (\delta \mathbf{k}) \omega^{-1} R_0(\zeta) \omega^{-1} (\delta \mathbf{k})^* g (\mathbf{D} + \mathbf{k}^\circ) \omega^{-1} \frac{1}{\lambda_0 - \zeta} P
		\\
		&+ \Bigl( (\mathbf{D} + \mathbf{k}^\circ) \omega^{-1} \frac{1}{\lambda_0 - \zeta^*} P \Bigr)^* g (\delta \mathbf{k}) \omega^{-1} \bigl((\mathbf{D} + \mathbf{k}^\circ) \omega^{-1} R_0(\zeta^*) \bigr)^* g (\delta \mathbf{k}) \omega^{-1} \frac{1}{\lambda_0 - \zeta} P
		\\
		&- \frac{1}{\lambda_0 - \zeta} P \omega^{-1} (\delta \mathbf{k})^* g (\delta \mathbf{k}) \omega^{-1} \frac{1}{\lambda_0 - \zeta} P,
	\end{aligned}
	\\
	&T^{\times} (\delta \mathbf{k}, \zeta) = P T (\delta \mathbf{k}, \zeta) P^\perp, \qquad T^{\perp} (\delta \mathbf{k}, \zeta) = P^\perp T (\delta \mathbf{k}, \zeta) P^\perp.
\end{align*}
Substituting~\eqref{T_decomp} into~\eqref{G2_integral}, we obtain
\begin{gather}
	\label{G2}
	G_2(\delta \mathbf{k}) = G^\circ_{2} (\delta \mathbf{k}) + G^{\times}_{2} (\delta \mathbf{k}) + G^{\times}_{2}(\delta \mathbf{k})^* + G^\perp_{2} (\delta \mathbf{k}),
	\\
	\label{G2^r_integral}
	G^r_{2} (\delta \mathbf{k}) = \frac{-1}{2 \pi i} \ointctrclockwise_\gamma \zeta T^r(\delta \mathbf{k}, \zeta) \, d\zeta, \qquad r \in \{\circ, \times, \perp\}.
\end{gather}
The operator $G^\circ_{2} (\delta \mathbf{k})$ acts in $\mathfrak{N}$; using the elementary equality $\frac{\zeta}{(\zeta - \lambda_0)^2} = \frac{1}{\zeta - \lambda_0} + \frac{\lambda_0}{(\zeta - \lambda_0)^2}$ and the formula for the derivative of Cauchy integral, we get
\begin{multline*}
	G^\circ_{2} (\delta \mathbf{k}) = P \omega^{-1} (\delta \mathbf{k})^* g (\delta \mathbf{k}) \omega^{-1} P
	\\
	- P \omega^{-1} (\delta \mathbf{k})^* g (\mathbf{D} + \mathbf{k}^\circ) \omega^{-1} \bigl(R_0^\perp(\lambda_0) + \lambda_0 (R_0^\perp)'(\lambda_0)\bigr)  \omega^{-1} (\delta \mathbf{k})^* g (\mathbf{D} + \mathbf{k}^\circ) \omega^{-1} P
	\\
	- P \omega^{-1} (\delta \mathbf{k})^* g (\mathbf{D} + \mathbf{k}^\circ) \omega^{-1} \bigl( (\mathbf{D} + \mathbf{k}^\circ) \omega^{-1} \bigl(R_0^\perp(\lambda_0) + \lambda_0 (R_0^\perp)'(\lambda_0)\bigr)\bigr)^* g (\delta \mathbf{k}) \omega^{-1} P
	\\
	- \bigl( (\mathbf{D} + \mathbf{k}^\circ) \omega^{-1} P \bigr)^* g (\delta \mathbf{k}) \omega^{-1} \bigl(R_0^\perp(\lambda_0) + \lambda_0 (R_0^\perp)'(\lambda_0)\bigr) \omega^{-1} (\delta \mathbf{k})^* g (\mathbf{D} + \mathbf{k}^\circ) \omega^{-1} P
	\\
	- \bigl( (\mathbf{D} + \mathbf{k}^\circ) \omega^{-1} P \bigr)^* g (\delta \mathbf{k}) \omega^{-1} \bigl( (\mathbf{D} + \mathbf{k}^\circ) \omega^{-1} \bigl(R_0^\perp(\lambda_0) + \lambda_0 (R_0^\perp)'(\lambda_0)\bigr) \bigr)^* g (\delta \mathbf{k}) \omega^{-1} P.
\end{multline*}
Here $(R_0^\perp)'(\lambda_0) \coloneqq \left. \frac{d}{d \lambda} R_0^\perp(\lambda) \right|_{\lambda = \lambda_0}$. From the first resolvent identity it directly follows that $(R_0^\perp)'(\lambda_0) = R_0^\perp(\lambda_0)^2$. Next, we have
\begin{equation}
	\label{G2times_P=0}
	G^\times_{2} (\delta \mathbf{k}) P = 0, \qquad P G^\times_{2} (\delta \mathbf{k})^* = 0, \qquad P G^\perp_{2} (\delta \mathbf{k}) = 0,
\end{equation}
and
\begin{equation}
	\label{frakG2circ}
	\begin{multlined}[c][0.8\textwidth]
		G^\circ_{2} (\delta \mathbf{k}) + \lambda_0 F^\times_1(\delta \mathbf{k}) F^\times_1(\delta \mathbf{k})^* = P \omega^{-1} (\delta \mathbf{k})^* g (\delta \mathbf{k}) \omega^{-1} P
		\\
		- P \omega^{-1} (\delta \mathbf{k})^* g (\mathbf{D} + \mathbf{k}^\circ) \omega^{-1} R_0^\perp(\lambda_0)  \omega^{-1} (\delta \mathbf{k})^* g (\mathbf{D} + \mathbf{k}^\circ) \omega^{-1} P
		\\
		- P \omega^{-1} (\delta \mathbf{k})^* g (\mathbf{D} + \mathbf{k}^\circ) \omega^{-1} \bigl( (\mathbf{D} + \mathbf{k}^\circ) \omega^{-1} R_0^\perp(\lambda_0)\bigr)^* g (\delta \mathbf{k}) \omega^{-1} P
		\\
		- \bigl( (\mathbf{D} + \mathbf{k}^\circ) \omega^{-1} P \bigr)^* g (\delta \mathbf{k}) \omega^{-1} R_0^\perp(\lambda_0) \omega^{-1} (\delta \mathbf{k})^* g (\mathbf{D} + \mathbf{k}^\circ) \omega^{-1} P
		\\
		- \bigl( (\mathbf{D} + \mathbf{k}^\circ) \omega^{-1} P \bigr)^* g (\delta \mathbf{k}) \omega^{-1} \bigl( (\mathbf{D} + \mathbf{k}^\circ) \omega^{-1} R_0^\perp(\lambda_0) \bigr)^* g (\delta \mathbf{k}) \omega^{-1} P
		 \eqqcolon \mathfrak{G}^\circ_2 (\delta \mathbf{k}),
	\end{multlined}
\end{equation}
to obtain~\eqref{frakG2circ} we have used~\eqref{F1times_F1times*}.
Moreover, we need to estimate the operator $G^\times_{2}(\delta \mathbf{k}) F(\mathbf{k})$. According to~\eqref{Y_Y0_est}, \eqref{X0_est}, \eqref{X0_check_est}, \eqref{Y_check_est}, and~\eqref{T}, the operator $T(\delta \mathbf{k}, \zeta)$ satisfies the estimate 
\begin{equation*}
	\| T(\delta \mathbf{k}, \zeta) \| \le (3 C_1 C_2^2 C_4 + C_1^2 C_3 + C_1^2) |\delta \mathbf{k}|^2, \qquad |\delta \mathbf{k}| \le \varkappa, \quad \zeta \in \gamma,
\end{equation*}
whence, by~\eqref{G2^r_integral},
\begin{gather}
	\label{G^times_2_est}
	\| G^\times_{2} (\delta \mathbf{k})\| \le C_{10} |\delta \mathbf{k}|^2, \qquad |\delta \mathbf{k}| \le \varkappa;
	\\
	\notag
	C_{10} = (2 \pi)^{-1} (\lambda_0 + d_0/2) l_\gamma (3 C_1 C_2^2 C_4 + C_1^2 C_3 + C_1^2) .
\end{gather}
Using~\eqref{F(k)_threshold_1}, the first relation~\eqref{G2times_P=0}, and~\eqref{G^times_2_est}, we arrive at
\begin{equation}
	\label{G^times_2_F_est}
	\| G^\times_{2} (\delta \mathbf{k}) F(\mathbf{k}) \| = \| G^\times_{2}(\delta \mathbf{k}) (F(\mathbf{k}) - P) \| \le C_7 C_{10} |\delta \mathbf{k}|^3, \qquad |\delta \mathbf{k}| \le \varkappa. 
\end{equation}
In the end of this section, we give an estimate for the remainer $\Xi(\delta \mathbf{k}, \mathbf{k})$. From~\eqref{scrR1_scrR2_est}, \eqref{Xi_integral} it follows that
\begin{equation}
	\label{Xi_est}
	\| \Xi(\delta \mathbf{k}, \mathbf{k})\| \le (2 \pi)^{-1} (\lambda_0 + d_0/2) l_\gamma C_6 |\delta \mathbf{k}|^3, \qquad |\delta \mathbf{k}| \le \varkappa.
\end{equation}

\subsection{Approximations for the operator exponential} 
We put
\begin{equation}
	\label{frakG^circ}
	\mathfrak{G}^\circ(\delta \mathbf{k}) \coloneqq  \lambda_0 P + \mathfrak{G}^\circ_1 (\delta \mathbf{k}) + \mathfrak{G}^\circ_2 (\delta \mathbf{k}), \qquad \mathfrak{G}^\circ(\delta \mathbf{k}) \colon \mathfrak{N} \to \mathfrak{N}.
\end{equation}
Here the operators $\mathfrak{G}^\circ_1 (\delta \mathbf{k})$ and $\mathfrak{G}^\circ_2 (\delta \mathbf{k})$ were defined in~\eqref{frakG1circ} and~\eqref{frakG2circ}.
In this section, we want to approximate the operator $e^{-i \tau \mathcal{A}(\mathbf{k})}P$, $\tau \in \mathbb{R}$, by $e^{-i \tau \mathfrak{G}^\circ(\delta \mathbf{k}) P}P$. Consider the difference
\begin{multline*}
	\bigl(e^{-i \tau \mathcal{A}(\mathbf{k})} - e^{-i \tau \mathfrak{G}^\circ(\delta \mathbf{k}) P}\bigr) P
	\\ = P \bigl(e^{-i \tau \mathcal{A}(\mathbf{k})} F(\mathbf{k}) - e^{-i \tau \mathfrak{G}^\circ(\delta \mathbf{k}) P}P\bigr) + e^{-i \tau \mathcal{A}(\mathbf{k})} (P - F(\mathbf{k})) + (F(\mathbf{k}) - P)e^{-i \tau \mathcal{A}(\mathbf{k})} F(\mathbf{k}).
\end{multline*}
By~\eqref{F(k)_threshold_1}, the last two terms admit the estimates 
\begin{equation*}
	\|e^{-i \tau \mathcal{A}(\mathbf{k})} (P - F(\mathbf{k}))\| \le C_7 |\delta \mathbf{k}|, \quad \|(F(\mathbf{k}) - P)e^{-i \tau \mathcal{A}(\mathbf{k})} F(\mathbf{k})\| \le C_7 |\delta \mathbf{k}|, \quad |\delta \mathbf{k}| \le \varkappa.
\end{equation*}
Next (cf.~\cite[the proof of Theorem~2.1]{BSu2008}),
\begin{equation*}
	P \bigl(e^{-i \tau \mathcal{A}(\mathbf{k})} F(\mathbf{k}) - e^{-i \tau \mathfrak{G}^\circ(\delta \mathbf{k})P} P\bigr) = P e^{-i \tau \mathfrak{G}^\circ(\delta \mathbf{k})P} \Sigma(\mathbf{k}, \tau),
\end{equation*}
where $\Sigma(\mathbf{k}, \tau) \coloneqq e^{i \tau \mathfrak{G}^\circ(\delta \mathbf{k})P} F (\mathbf{k}) e^{-i \tau \mathcal{A}(\mathbf{k})} - P$. We have
\begin{equation*}
	\Sigma(\mathbf{k}, \tau) = \Sigma(\mathbf{k}, 0) +  \int_0^\tau \Sigma' (\mathbf{k}, \tilde{\tau}) d \tilde{\tau}.
\end{equation*}
Obviously, $\Sigma(\mathbf{k}, 0) = F (\mathbf{k}) - P$, and, by~\eqref{F(k)_threshold_1}, $\| P e^{-i \tau \mathfrak{G}^\circ(\delta \mathbf{k})P} \Sigma(\mathbf{k}, 0)\| \le C_7 |\delta \mathbf{k}|$, $|\delta \mathbf{k}| \le \varkappa$. Next,
\begin{equation*}
	\Sigma' (\mathbf{k}, \tau) \coloneqq \frac{d \Sigma}{d \tau} (\mathbf{k}, \tau) = - i e^{i \tau \mathfrak{G}^\circ(\delta \mathbf{k})P} \bigl(\mathcal{A}(\mathbf{k}) F (\mathbf{k}) - \mathfrak{G}^\circ(\delta \mathbf{k})P\bigr) F(\mathbf{k})  e^{-i \tau \mathcal{A}(\mathbf{k})} F(\mathbf{k}).
\end{equation*}
Consider the operator $P \bigl(\mathcal{A}(\mathbf{k}) F (\mathbf{k}) - \mathfrak{G}^\circ(\delta \mathbf{k}) P \bigr) F(\mathbf{k})$. Using the second identity in~\eqref{F1timesP=0}, \eqref{F1times_F}, \eqref{AF_threshold}, \eqref{G1}, \eqref{G2}, the second and the third identities in~\eqref{G2times_P=0}, \eqref{frakG2circ}, and \eqref{frakG^circ}, we have
\begin{multline*}
	P \bigl(\mathcal{A}(\mathbf{k}) F (\mathbf{k}) - \mathfrak{G}^\circ(\delta \mathbf{k}) P\bigr) F(\mathbf{k}) \\
	= P \bigl(\mathcal{A}(\mathbf{k}) F (\mathbf{k}) - \lambda_0 P - G_1(\delta \mathbf{k})  - G_2(\delta \mathbf{k})\bigr) F(\mathbf{k})  + \lambda_0 F^\times_1(\delta \mathbf{k}) \Phi(\delta \mathbf{k}, \mathbf{k}) F(\mathbf{k}) + G^\times_2(\mathbf{k}) F(\mathbf{k})
	\\
	= P \Xi(\delta \mathbf{k}, \mathbf{k}) F(\mathbf{k})  + \lambda_0 F^\times_1(\delta \mathbf{k}) \Phi(\delta \mathbf{k}, \mathbf{k}) F(\mathbf{k}) + G^\times_2(\mathbf{k}) F(\mathbf{k}),
\end{multline*}
whence, by~\eqref{F1times_Phi_F_est}, \eqref{G^times_2_F_est}, \eqref{Xi_est},
\begin{gather*}
	\begin{alignedat}{2}
		\bigl\| P \bigl(\mathcal{A}(\mathbf{k}) F (\mathbf{k}) - \mathfrak{G}^\circ(\delta \mathbf{k}) P\bigr) F(\mathbf{k}) \bigr\| &\le C_{11}  |\delta \mathbf{k}|^3, &\qquad &|\delta \mathbf{k}| \le \varkappa;
		\\
		\biggl\| P e^{-i \tau \mathfrak{G}^\circ(\delta \mathbf{k})P} \int_0^\tau \Sigma' (\mathbf{k}, \tilde{\tau}) \, d \tilde{\tau}\biggr\| &\le  C_{11} |\tau| |\delta \mathbf{k}|^3, &\qquad &|\delta \mathbf{k}| \le \varkappa;
	\end{alignedat}
	\\
	C_{11} = \lambda_0 C_9 + C_7 C_{10} + (2 \pi)^{-1} (\lambda_0 + d_0/2) l_\gamma C_6. 
\end{gather*}

Recalling the expressions for the constants, from what has been said above we deduce the following result.
\begin{thrm}
	\label{exp_main_thrm}
	Let $\tau \in \mathbb{R}$ and $|\delta \mathbf{k}| \le \varkappa$. We have
	\begin{equation*}
		\bigl\| \bigl(e^{-i \tau \mathcal{A}(\mathbf{k})} - e^{-i \tau \mathfrak{G}^\circ(\delta \mathbf{k}) P}\bigr) P \bigr\| \le 3C_7 |\delta \mathbf{k}| + C_{11} |\tau| |\delta \mathbf{k}|^3,
	\end{equation*}
	where the constants $C_7$ and $C_{11}$ are defined by
	\begin{gather*}
		C_7 = (2 \pi)^{-1} l_\gamma (C_1 C_2 + C_1 \check{C}_2 + C_1^2 \varkappa) ,
		\\
		\begin{multlined}[c][0.8\textwidth]
			C_{11} = 2^{-1} \pi^{-2} \lambda_0 l_\gamma^2 C_1 C_2 \bigl(2 C_1 C_2 \check{C}_2 C_4 + C_1^2 C_3 + 2 C_1^2 C_2 C_4 \varkappa + C_1 C_2^2 C_4 + C_1^2\bigr) \\ + (2 \pi)^{-2} (\lambda_0 + d_0/2) l_\gamma^2 \bigl(C_1 C_2 + C_1 \check{C}_2 + C_1^2 \varkappa\bigr) \bigl(3 C_1 C_2^2 C_4 + C_1^2 C_3 + C_1^2\bigr) 
			\\ 
			+ (2 \pi)^{-1} (\lambda_0 + d_0/2) l_\gamma \bigl(3 C_1^2 C_2 C_4 
			+ 3 C_1 C_2^2 \check{C}_2 C_4^2 
			+ 3 C_1^2 C_2 C_3 C_4 
			+ 3 C_1^2 C_2^2 C_4^2 \varkappa 
			\\
			+ C_1^2 \check{C}_2 C_3 C_4 
			+ C_1^3 C_3 C_4 \varkappa
			+ C_1 C_2^3 C_4^2 
			+ C_1^2 \check{C}_2 C_4   
			+ C_1^3 C_4 \varkappa\bigr).
		\end{multlined}
	\end{gather*}
\end{thrm}

\subsection{Calculation of the operator $\mathfrak{G}^\circ(\delta \mathbf{k})$ in a basis of $\mathfrak{N}$}
Let $\{\varsigma_p\}_{p=1}^n$ be an orthonormal basis in $\mathfrak{N}$. In this section, our aim is to calculate the matrix elements of operator~\eqref{frakG^circ} in this basis. 

First of all, obviously, $(\lambda_0 P \varsigma_p, \varsigma_l) = \lambda_0 \delta_{lp}$. Let us proceed to calculation of $(\mathfrak{G}^\circ_1 (\delta \mathbf{k}) \varsigma_p, \varsigma_l)$. We have
\begin{align*}
	(\mathbf{D} + \mathbf{k}^\circ) \omega^{-1} \varsigma_p &= -i \bigl\{ \bigl(\tfrac{\partial}{\partial x_1} + i k^\circ_1 \bigr ) (\omega^{-1} \varsigma_p), \ldots, \bigl(\tfrac{\partial}{\partial x_d} + i k^\circ_d \bigr ) (\omega^{-1} \varsigma_p) \bigr\}^\mathrm{t},
	\\
	(\delta \mathbf{k}) \omega^{-1} \varsigma_p &= \bigl\{ (\delta k_1) \omega^{-1} \varsigma_p, \ldots, (\delta k_d) \omega^{-1} \varsigma_p \bigr\}^\mathrm{t}, \qquad p = 1,\ldots,n,
\end{align*}
whence
\begin{gather}
	\label{V_f1}
	\begin{multlined}[c][0.7\textwidth]
		 (P \omega^{-1} (\delta \mathbf{k})^* g (\mathbf{D} + \mathbf{k}^\circ) \omega^{-1} P \varsigma_p, \varsigma_l) \\ = (g (\mathbf{D} + \mathbf{k}^\circ) \omega^{-1} P \varsigma_p, (\delta \mathbf{k}) \omega^{-1} P \varsigma_l)  = i \sum_{r=1}^{d} (\delta k_r) \tilde{\mathfrak{g}}^{1,lp}_r, \quad p,l = 1, \ldots,n,
	\end{multlined}
	\\
	\notag
	\tilde{\mathfrak{g}}^{1,lp}_r \coloneqq - \sum_{s=1}^{d} \int_{\Omega} g_{rs}(\mathbf{x}) \omega(\mathbf{x})^{-1} \varsigma_l (\mathbf{x})^*  \left(\tfrac{\partial}{\partial x_s} + ik^\circ_s\right) \bigl(\omega(\mathbf{x})^{-1} \varsigma_p(\mathbf{x})\bigr) \, d\mathbf{x}.
\end{gather}
Next, since the operator $\bigl((\mathbf{D} + \mathbf{k}^\circ) \omega^{-1} P \bigr)^* g (\delta \mathbf{k}) \omega^{-1} P$  is adjoint to $P \omega^{-1} (\delta \mathbf{k})^* g (\mathbf{D} + \mathbf{k}^\circ) \omega^{-1} P$, it is seen that
\begin{equation}
	\label{V_f2}
	\bigl(\bigl((\mathbf{D} + \mathbf{k}^\circ) \omega^{-1} P \bigr)^* g (\delta \mathbf{k}) \omega^{-1} P \varsigma_p, \varsigma_l \bigr) = -i \sum_{r=1}^{d} (\delta k_r) (\tilde{\mathfrak{g}}^{1,pl}_r)^*,  \qquad p,l = 1, \ldots,n.
\end{equation}
Thus, from~\eqref{frakG1circ}, \eqref{V_f1}, \eqref{V_f2} it follows that
\begin{equation}
	\label{frak_g_1}
	(\mathfrak{G}^\circ_1 (\delta \mathbf{k})\varsigma_p, \varsigma_l) = \left\langle \mathfrak{g}^{1,lp}, \delta \mathbf{k} \right\rangle, \qquad p,l = 1, \ldots,n,
\end{equation} 
where $\mathfrak{g}^{1,lp} = (\mathfrak{g}^{1,lp}_{1}, \ldots, \mathfrak{g}^{1,lp}_{d})^\mathrm{t}$ is the column-vector with the entries
\begin{multline*}
	 \mathfrak{g}^{1,lp}_{r} = i \bigl(\tilde{\mathfrak{g}}^{1,lp}_r -  (\tilde{\mathfrak{g}}^{1,pl}_r)^* \bigr)
	 \\ 
	 = i \sum_{s=1}^{d} \int_{\Omega} g_{rs}(\mathbf{x}) \left( \omega(\mathbf{x})^{-1} \varsigma_p (\mathbf{x})  \frac{\partial}{\partial x_s}\bigl(\omega(\mathbf{x})^{-1} \varsigma_l(\mathbf{x})^*\bigr)  - \omega(\mathbf{x})^{-1} \varsigma_l (\mathbf{x})^*  \frac{\partial}{\partial x_s} \bigl(\omega(\mathbf{x})^{-1} \varsigma_p(\mathbf{x})\bigr) \right) \, d\mathbf{x} \\ 
	 + 2  \sum_{s=1}^{d} k^\circ_s \int_{\Omega} g_{rs}(\mathbf{x}) \omega(\mathbf{x})^{-2}   \varsigma_p(\mathbf{x}) \varsigma_l (\mathbf{x})^*  d\mathbf{x}, \qquad  r =1, \ldots,d.
\end{multline*}
Let us proceed to calculation of the matrix elements of the operator $\mathfrak{G}^\circ_2 (\delta \mathbf{k})$, defined by~\eqref{frakG2circ}. It is convenient to write this operator as
\begin{multline*}
	\mathfrak{G}^\circ_2 (\delta \mathbf{k}) = - \bigl( P \omega^{-1} (\delta \mathbf{k})^* g (\mathbf{D} + \mathbf{k}^\circ) \omega^{-1} + \bigl( (\mathbf{D} + \mathbf{k}^\circ) \omega^{-1} P \bigr)^* g (\delta \mathbf{k}) \omega^{-1} \bigr)
	\\ 
	\times \left( R_0^\perp(\lambda_0)  \omega^{-1} (\delta \mathbf{k})^* g (\mathbf{D} + \mathbf{k}^\circ) \omega^{-1} P + \bigl( (\mathbf{D} + \mathbf{k}^\circ) \omega^{-1} R_0^\perp(\lambda_0)\bigr)^* g (\delta \mathbf{k}) \omega^{-1} P \right) 
	\\ 
	+ P \omega^{-1} (\delta \mathbf{k})^* g (\delta \mathbf{k}) \omega^{-1} P. 
\end{multline*}
Denote by $\Lambda^p (\delta \mathbf{k})$ the result of the action of the operator $$\left( R_0^\perp(\lambda_0)  \omega^{-1} (\delta \mathbf{k})^* g (\mathbf{D} + \mathbf{k}^\circ) \omega^{-1} P + \bigl( (\mathbf{D} + \mathbf{k}^\circ) \omega^{-1} R_0^\perp(\lambda_0)\bigr)^* g (\delta \mathbf{k}) \omega^{-1} P \right)$$
on the basis element $\varsigma_p$. Obviously, $\Lambda^p (\delta \mathbf{k}) \in \widetilde{\mathcal{H}}^1(\Omega)$ is the   (weak) solution of the equation 
\begin{equation*}
	\bigl( \mathcal{A}(\mathbf{k}^\circ) - \lambda_0 I \bigr) \Lambda^p (\delta \mathbf{k}) = \omega^{-1} (\delta \mathbf{k})^* g (\mathbf{D} + \mathbf{k}^\circ) \omega^{-1} \varsigma_p +  \omega^{-1} (\mathbf{D} + \mathbf{k}^\circ)^* g (\delta \mathbf{k}) \omega^{-1} \varsigma_p, 
	\qquad \Lambda^p (\delta \mathbf{k}) \perp \mathfrak{N}.
\end{equation*}
The right-hand side of this equation is linear in $\delta \mathbf{k}$ and has the following form:
\begin{multline*}
	-i \sum_{r,s=1}^{d} \omega(\mathbf{x})^{-1} (\delta k_r) g_{rs}(\mathbf{x})  \left(\frac{\partial}{\partial x_s} + ik^\circ_s\right) \bigl( \omega(\mathbf{x})^{-1} \varsigma_p(\mathbf{x}) \bigr) \\ - i \sum_{r,s=1}^{d} \omega(\mathbf{x})^{-1} \left(\frac{\partial}{\partial x_s} + ik^\circ_s\right) \bigl( g_{sr}(\mathbf{x})  (\delta k_r) \omega(\mathbf{x})^{-1} \varsigma_p(\mathbf{x})\bigr).
\end{multline*}
Therefore, $\Lambda^p (\delta \mathbf{k})$ can be written as $\Lambda^p (\delta \mathbf{k}) = -i \sum_{r=1}^{d} (\delta k_r) \Lambda^p_r$, where $\Lambda^p_r \in \widetilde{\mathcal{H}}^1(\Omega)$ is the solution of the equation
\begin{multline*}
	\bigl( \mathcal{A}(\mathbf{k}^\circ) - \lambda_0 I \bigr) \Lambda^p_{r}(\mathbf{x})  = \sum_{s=1}^{d} \omega(\mathbf{x})^{-1} g_{rs}(\mathbf{x})  \left(\frac{\partial}{\partial x_s} + ik^\circ_s\right) \bigl( \omega(\mathbf{x})^{-1} \varsigma_p(\mathbf{x}) \bigr) \\ + \sum_{s=1}^{d} \omega(\mathbf{x})^{-1} \left(\frac{\partial}{\partial x_s} + ik^\circ_s\right) \bigl( g_{sr}(\mathbf{x}) \omega(\mathbf{x})^{-1} \varsigma_p(\mathbf{x}) \bigr), 
	\quad \Lambda^p_r \perp \mathfrak{N}.
\end{multline*}
Next, let us calculate the inner product
\begin{equation*}
	- \left( \bigl( P \omega^{-1} (\delta \mathbf{k})^* g (\mathbf{D} + \mathbf{k}^\circ) \omega^{-1} + \bigl( (\mathbf{D} + \mathbf{k}^\circ) \omega^{-1} P \bigr)^* g (\delta \mathbf{k}) \omega^{-1} \bigr) \Lambda^p(\delta \mathbf{k}), \varsigma_l \right), \qquad p,l = 1, \ldots,n.
\end{equation*}
Similarly to~\eqref{V_f1}, \eqref{V_f2}, one obtains
\begin{multline*}
	- \left( \bigl( P \omega^{-1} (\delta \mathbf{k})^* g (\mathbf{D} + \mathbf{k}^\circ) \omega^{-1} + \bigl( (\mathbf{D} + \mathbf{k}^\circ) \omega^{-1} P \bigr)^* g (\delta \mathbf{k}) \omega^{-1} \bigr) \Lambda^p(\delta \mathbf{k}) , \varsigma_l \right) = - \sum_{q,r,s=1}^{d} (\delta k_r) (\delta k_q) \\ \times \int_{\Omega} g_{qs}(\mathbf{x}) \left( \omega(\mathbf{x})^{-1} \Lambda^p_r (\mathbf{x})  \frac{\partial}{\partial x_s} \bigl(\omega(\mathbf{x})^{-1} \varsigma_l(\mathbf{x})^* \bigr)  - \omega(\mathbf{x})^{-1} \varsigma_l (\mathbf{x})^*  \frac{\partial}{\partial x_s} \bigl(\omega(\mathbf{x})^{-1}  \Lambda^p_r(\mathbf{x}) \bigr) \right) \, d\mathbf{x} \\ 
	+ 2 i \sum_{q,r,s=1}^{d} (\delta k_r) (\delta k_q) k^\circ_s \int_{\Omega} g_{qs}(\mathbf{x}) \omega(\mathbf{x})^{-2}   \Lambda^p_r(\mathbf{x}) \varsigma_l (\mathbf{x})^*  d\mathbf{x}, \qquad p,l=1,\ldots,n.
\end{multline*}
Finally,
\begin{multline*}
	(P \omega^{-1} (\delta \mathbf{k})^* g (\delta \mathbf{k}) \omega^{-1} P \varsigma_p, \varsigma_l) = (g (\delta \mathbf{k}) \omega^{-1} P \varsigma_p, (\delta \mathbf{k}) \omega^{-1} P  \varsigma_l) \\ 
	= \sum_{r,q=1}^{d} (\delta k_r) (\delta k_q) \int_{\Omega} g_{qr}(\mathbf{x}) \omega(\mathbf{x})^{-2} \varsigma_p(\mathbf{x}) \varsigma_l(\mathbf{x})^* d\mathbf{x},\qquad p,l=1,\ldots,n.
\end{multline*}
Thereby, we have obtained the formula for the matrix elements of the operator $\mathfrak{G}^\circ_2 (\delta \mathbf{k})$:
\begin{equation*}
	(\mathfrak{G}^\circ_2 (\delta \mathbf{k}) \varsigma_p, \varsigma_l) = \left\langle \mathfrak{g}^{2,lp} (\delta \mathbf{k}), (\delta \mathbf{k}) \right\rangle, \qquad p,l=1,\ldots,n,
\end{equation*}
where $\mathfrak{g}^{2,lp}$ is $(d \times d)$-matrix with the entries
\begin{multline*}
	\mathfrak{g}^{2,lp}_{rq} = - \sum_{s=1}^{d} \int_{\Omega} g_{qs}(\mathbf{x}) \left( \omega(\mathbf{x})^{-1} \Lambda^p_r (\mathbf{x})  \frac{\partial}{\partial x_s}(\omega(\mathbf{x})^{-1} \varsigma_l(\mathbf{x})^*)  - \omega(\mathbf{x})^{-1} \varsigma_l (\mathbf{x})^*  \frac{\partial}{\partial x_s} (\omega(\mathbf{x})^{-1}  \Lambda^p_r(\mathbf{x})) \right)  d\mathbf{x} \\ 
	+ 2 i \sum_{s=1}^{d} k^\circ_s \int_{\Omega} g_{qs}(\mathbf{x}) \omega(\mathbf{x})^{-2}   \Lambda^p_r(\mathbf{x}) \varsigma_l (\mathbf{x})^*  d\mathbf{x}
	+
	\int_{\Omega} g_{qr}(\mathbf{x}) \omega(\mathbf{x})^{-2} \varsigma_p(\mathbf{x}) \varsigma_l(\mathbf{x})^* d\mathbf{x}, \qquad r,q = 1, \ldots, d.
\end{multline*}
As a result, operator~\eqref{frakG^circ} is represented in the basis $\{\varsigma_p\}_{p=1}^n$ by the matrix
\begin{equation*}
	\mathfrak{g}(\delta \mathbf{k}) \coloneqq \bigl\{\lambda_0 \delta_{lp} +  \left\langle \mathfrak{g}^{1,lp}, \delta \mathbf{k} \right\rangle + \left\langle \mathfrak{g}^{2,lp} (\delta \mathbf{k}), (\delta \mathbf{k}) \right\rangle \bigr\}_{l,p=1}^n.
\end{equation*}

In the end of this section, consider the action of the exponential $e^{-i \tau \mathfrak{G}^\circ(\delta \mathbf{k})}$ on the element~$\varsigma_j$. It is easy to check that
\begin{equation}
	\label{exp_basis}
	e^{-i \tau \mathfrak{G}^\circ(\delta \mathbf{k})} \varsigma_j = \sum_{l=1}^{n} c_{jl}(\tau) \varsigma_l,
\end{equation} 
where $ \{c_{j1}(\tau), \ldots, c_{jn}(\tau)\}^\mathrm{t} \eqqcolon \mathbf{c}_j(\tau) = e^{-i \tau \mathfrak{g}(\delta \mathbf{k})} \mathbf{e}_j$ is the solution of the system
\begin{equation*}
	\left\{
		i \frac{\partial}{\partial \tau} \mathbf{c}_{j}(\tau)  = (\mathfrak{g}(\delta \mathbf{k}) \mathbf{c}_j)(\tau), \qquad
		\mathbf{c}_{j}(0) = \mathbf{e}_j.
	\right.
\end{equation*}

\section{Main results of the paper}
\label{ch4}
In this section, we formulate the main results of the paper. Let $\varepsilon > 0$ be a small parameter. If $F(\mathbf{x})$ is a $\Gamma$-periodic function, then we put  $F^\varepsilon (\mathbf{x}) \coloneqq F(\varepsilon^{-1} \mathbf{x})$. In $L_2(\mathbb{R}^d)$, we consider the operator formally defined by the differential expression
\begin{equation}
	\label{A_eps}
	\mathcal{A}_\varepsilon = - \omega^\varepsilon(\mathbf{x})^{-1} \operatorname{div} g^\varepsilon(\mathbf{x}) \nabla \omega^\varepsilon(\mathbf{x})^{-1}, \quad g(\mathbf{x}) = \check{g}(\mathbf{x}) \omega(\mathbf{x})^2.
\end{equation}
Here $\check{g}$ and $\omega$ are $\Gamma$-periodic functions satisfying conditions~\eqref{g_check_cond}, \eqref{omega_cond1}, and $\omega(\mathbf{x}) > 0$; \mbox{$\omega, \omega^{-1} \in L_\infty$}. The precise definition of the operator $\mathcal{A}_\varepsilon$ is given in terms of the corresponding quadratic form (cf.~\eqref{a_form_2}). Operators~\eqref{A} and~\eqref{A_eps} satisfy the following relation:
\begin{equation*}
	\mathcal{A}_\varepsilon = \varepsilon^{-2} T_\varepsilon^* \mathcal{A} T_\varepsilon,
\end{equation*}
where $T_\varepsilon$ is the operator of scaling transformation: $(T_\varepsilon u)(\mathbf{x}) = \varepsilon^{d/2} u(\varepsilon \mathbf{x})$.

Let $\{\varsigma_j\}_{j=1}^n$ be an (arbitrary) orthonormal basis in $\mathfrak{N}$ and let $f_j \in L_2(\mathbb{R}^d)$, $j = 1, \ldots,n$. We suppose that the functions $\varsigma_j (\mathbf{x})$, $j=1,\ldots,n$, are $\Gamma$-periodically extended to $\mathbb{R}^d$. We study the behavior of the solutions $u_{j,\varepsilon} (\mathbf{x}, \tau)$, $j = 1, \ldots,n$, $\varepsilon \to 0$, of the following Cauchy problem for the nonstationary Schr\"{o}dinger equation
\begin{equation}
	\label{Cauchy_problem}
	\left\{
	\begin{aligned}
		&i \frac{\partial}{\partial \tau} u_{j,\varepsilon} (\mathbf{x},\tau) = (\mathcal{A}_\varepsilon u_{j,\varepsilon})(\mathbf{x},\tau),\\
		&u_{j,\varepsilon} (\mathbf{x},0) = e^{i \varepsilon^{-1} \left\langle \mathbf{k}^\circ, \mathbf{x} \right\rangle } \varsigma_j^\varepsilon(\mathbf{x}) f_j(\mathbf{x}).
	\end{aligned}
	\right.
\end{equation} 

In $L_2(\mathbb{R}^d;\mathbb{C}^n)$, we consider the operator
\begin{align}
	\label{A^eff_def}
	&\mathcal{A}^{\mathrm{eff}}_\varepsilon = \begin{pmatrix}
		\mathcal{A}^{\mathrm{eff},11}_\varepsilon & \cdots & \mathcal{A}^{\mathrm{eff},1n}_\varepsilon \\
		\vdots & \ddots & \vdots \\
		\mathcal{A}^{\mathrm{eff},n1}_\varepsilon & \cdots & \mathcal{A}^{\mathrm{eff},nn}_\varepsilon 
	\end{pmatrix}, \qquad \Dom \mathcal{A}^{\mathrm{eff}}_\varepsilon = H^2(\mathbb{R}^d; \mathbb{C}^n),
	\\
	\label{A^eff_lp_def}
	&\mathcal{A}^{\mathrm{eff}, lp}_\varepsilon \coloneqq \varepsilon^{-2} \lambda_0 I - i \varepsilon^{-1} \left\langle \mathfrak{g}^{1,lp}, \nabla \right\rangle - \operatorname{div} \mathfrak{g}^{2,lp} \nabla,
\end{align}
which is called \emph{the effective operator}. Let $\mathbf{v}^{\mathrm{eff}}_{j, \varepsilon} (\mathbf{x},\tau)$ be the solution of the corresponding "homogenized" problem
\begin{equation}
	\label{Cauchy_problem_eff}
	\left\{
	\begin{aligned}
		&i \frac{\partial}{\partial \tau} \mathbf{v}^{\mathrm{eff}}_{j, \varepsilon} (\mathbf{x},\tau) = \mathcal{A}^{\mathrm{eff}}_\varepsilon \mathbf{v}^{\mathrm{eff}}_{j, \varepsilon} (\mathbf{x},\tau) ,\\
		&\mathbf{v}^{\mathrm{eff}}_{j, \varepsilon} (\mathbf{x},0) = J_j f_j (\mathbf{x}).
	\end{aligned}
	\right.
\end{equation}
Here the operator $J_j \colon \mathbb{C} \to \mathbb{C}^n$ is defined by the rule $a \mapsto a \mathbf{e}_j$. 
Put
\begin{equation}
	\label{v_eff}
	u^{\mathrm{eff}}_{j,\varepsilon} (\mathbf{x},\tau) \coloneqq e^{i \varepsilon^{-1} \left\langle \mathbf{k}^\circ, \mathbf{x} \right\rangle } \sum_{l=1}^{n} \varsigma_l^\varepsilon (\mathbf{x}) v^{\mathrm{eff}}_{jl, \varepsilon} (\mathbf{x},\tau).	
\end{equation}
The solutions of problems~\eqref{Cauchy_problem}, \eqref{Cauchy_problem_eff} and $u^{\mathrm{eff}}_{j,\varepsilon}$ can be represented as follows:
\begin{equation}
	\label{operator_represent}
	\begin{gathered}
		u_{j,\varepsilon} (\cdot, \tau) = e^{-i \tau \mathcal{A}_\varepsilon} e^{i \varepsilon^{-1} \left\langle \mathbf{k}^\circ, \cdot \right\rangle } \varsigma_j^\varepsilon f_j, \qquad \mathbf{v}^{\mathrm{eff}}_{j, \varepsilon} (\cdot,\tau) = e^{-i \tau \mathcal{A}^{\mathrm{eff}}_\varepsilon} J_j f_j, 
		\\
		u^{\mathrm{eff}}_{j,\varepsilon} (\cdot,\tau) =  e^{i \varepsilon^{-1} \left\langle \mathbf{k}^\circ, \cdot \right\rangle } \sum_{l=1}^{n} \varsigma_l^\varepsilon \tilde{J}_l e^{-i \tau \mathcal{A}^{\mathrm{eff}}_\varepsilon} J_j f_j,
	\end{gathered}
\end{equation}
where $\tilde{J}_l \colon \mathbb{C}^n \to \mathbb{C}$ is defined by the formula $\tilde{J}_l \mathbf{c} = \left\langle \mathbf{c}, \mathbf{e}_l\right\rangle$.

\begin{thrm}
	\label{main_thrm}
	Let $u_\varepsilon$ be the solution of problem~\eqref{Cauchy_problem}, and let $u^{\mathrm{eff}}_{j,\varepsilon}$ be defined by~\eqref{v_eff}. Let $\varepsilon > 0$, $\tau \in \mathbb{R}$, $f_j \in H^3(\mathbb{R}^d)$. We have
	\begin{equation}
		\label{thrm_est}
		\| u_{j,\varepsilon} (\cdot, \tau) - u^{\mathrm{eff}}_{j,\varepsilon} (\cdot, \tau) \|_{L_2(\mathbb{R}^d)} \le \mathcal{C} (1+|\tau|) \varepsilon \|f_j\|_{H^3(\mathbb{R}^d)}, \qquad j = 1, \ldots,n,
	\end{equation}
	with the constant $\mathcal{C}$, which depends only on $\lambda_0$, $\varkappa$, $d_0$, $\|g\|_{L_\infty}$, $\| \omega^{-1} \|_{L_\infty}$, and $\| \varsigma_l\|_{L_\infty}$, $l=1,\ldots, n$.
\end{thrm}
\begin{proof}
	By~\eqref{operator_represent}, estimate~\eqref{thrm_est} can be reformulated in the operator terms:
	\begin{equation}
		\label{thrm_est_operators}
		\Bigl\| e^{-i \tau \mathcal{A}_\varepsilon} e^{i \varepsilon^{-1} \left\langle \mathbf{k}^\circ, \cdot \right\rangle } \varsigma_j^\varepsilon  - e^{i \varepsilon^{-1} \left\langle \mathbf{k}^\circ, \cdot \right\rangle } \sum_{l=1}^{n} \varsigma_l^\varepsilon \tilde{J}_l e^{-i \tau \mathcal{A}^{\mathrm{eff}}_\varepsilon} J_j  \Bigr\|_{H^3 (\mathbb{R}^d) \to L_2(\mathbb{R}^d)} \le \mathcal{C} (1+|\tau|) \varepsilon,
	\end{equation}
	where $\tau \in \mathbb{R}$, $\varepsilon > 0$. Thus, our aim is to prove~\eqref{thrm_est_operators}. Since the operator $(-\Delta + I)^{3/2}$ is an isometric isomorphism of the Sobolev space $H^3(\mathbb{R}^d)$ onto $L_2(\mathbb{R}^d)$, we have 
	\begin{multline*}
		\Bigl\| e^{-i \tau \mathcal{A}_\varepsilon} e^{i \varepsilon^{-1} \left\langle \mathbf{k}^\circ, \cdot \right\rangle } \varsigma_j^\varepsilon  - e^{i \varepsilon^{-1} \left\langle \mathbf{k}^\circ, \cdot \right\rangle } \sum_{l=1}^{n} \varsigma_l^\varepsilon \tilde{J}_l e^{-i \tau \mathcal{A}^{\mathrm{eff}}_\varepsilon} J_j  \Bigr\|_{H^3 (\mathbb{R}^d) \to L_2(\mathbb{R}^d)}
		\\
		=  \Bigl\| \Bigl(e^{-i \tau \mathcal{A}_\varepsilon} e^{i \varepsilon^{-1} \left\langle \mathbf{k}^\circ, \cdot \right\rangle } \varsigma_j^\varepsilon  - e^{i \varepsilon^{-1} \left\langle \mathbf{k}^\circ, \cdot \right\rangle } \sum_{l=1}^{n} \varsigma_l^\varepsilon \tilde{J}_l e^{-i \tau \mathcal{A}^{\mathrm{eff}}_\varepsilon} J_j \Bigr) (-\Delta + I)^{-3/2} \Bigr\|_{L_2 (\mathbb{R}^d) \to L_2(\mathbb{R}^d)}.
	\end{multline*}
	From the unitarity of the scaling transformation it directly follows that
	\begin{equation*}
		\begin{multlined}[c][0.9\textwidth]
			\Bigl\| \Bigl(e^{-i \tau \mathcal{A}_\varepsilon} e^{i \varepsilon^{-1} \left\langle \mathbf{k}^\circ, \cdot \right\rangle } \varsigma_j^\varepsilon  - e^{i \varepsilon^{-1} \left\langle \mathbf{k}^\circ, \cdot \right\rangle } \sum_{l=1}^{n} \varsigma_l^\varepsilon \tilde{J}_l e^{-i \tau \mathcal{A}^{\mathrm{eff}}_\varepsilon} J_j \Bigr) (-\Delta + I)^{-3/2} \Bigr\|_{L_2 (\mathbb{R}^d) \to L_2(\mathbb{R}^d)}
			\\ 
			= \Bigl\| \Bigl(e^{-i \tau \varepsilon^{-2} \mathcal{A}} e^{i  \left\langle \mathbf{k}^\circ, \cdot \right\rangle } \varsigma_j  - e^{i \left\langle \mathbf{k}^\circ, \cdot \right\rangle } \sum_{l=1}^{n} \varsigma_l \tilde{J}_l e^{-i \tau \varepsilon^{-2} \mathcal{A}^{\mathrm{eff}}} J_j \Bigr) \varepsilon^{3}  (-\Delta + \varepsilon^{2} I)^{-3/2} \Bigr\|_{L_2 (\mathbb{R}^d) \to L_2(\mathbb{R}^d)},
		\end{multlined}
	\end{equation*}
	where
	\begin{equation*}
		\mathcal{A}^{\mathrm{eff}} = \{\mathcal{A}^{\mathrm{eff},lp}\}_{l,p=1}^n, \qquad \mathcal{A}^{\mathrm{eff}, lp} \coloneqq  \lambda_0 I - i \left\langle \mathfrak{g}^{1,lp}, \nabla \right\rangle - \operatorname{div} \mathfrak{g}^{2,lp} \nabla.
	\end{equation*} 
	
	Next, we need the following operator identities:
	\begin{align}
			\label{Fourier_ident_1}
			\Phi^* \varepsilon^3 (|\delta \mathbf{k}|^2 + \varepsilon^2)^{-3/2} \Phi e^{i\left\langle \mathbf{k}^\circ, \mathbf{x}\right\rangle } &= e^{i\left\langle \mathbf{k}^\circ, \mathbf{x}\right\rangle } \varepsilon^{3} (- \Delta + \varepsilon^{2}I)^{-3/2},
			\\
			\label{Fourier_ident_2}
			\Phi^* \tilde{J}_l e^{-i \tau \varepsilon^{-2} \mathfrak{g} (\delta \mathbf{k})} J_j \varepsilon^3 (|\delta \mathbf{k}|^2 + \varepsilon^2)^{-3/2} \Phi e^{i\left\langle \mathbf{k}^\circ, \mathbf{x}\right\rangle } &= e^{i\left\langle \mathbf{k}^\circ, \mathbf{x}\right\rangle } \tilde{J}_l e^{-i \varepsilon^{-2} \tau \mathcal{A}^{\mathrm{eff}}} J_j \varepsilon^{3} (- \Delta + \varepsilon^{2}I)^{-3/2}.
	\end{align}
	Introduce the projection $F_\varkappa \coloneqq \Phi^* \chi_{\mathrm{B}_\varkappa(\mathbf{0})}(\mathbf{k}) \Phi$. Here $\chi_{\mathrm{B}_\varkappa(\mathbf{0})}(\mathbf{k})$ is the characteristic
	function of the ball $\mathrm{B}_\varkappa(\mathbf{0})$. Obviously, $\varepsilon^3 (|\delta \mathbf{k}|^2 + \varepsilon^2)^{-3/2} (1 - \chi_{\mathrm{B}_\varkappa(\mathbf{0})}(\delta \mathbf{k})) \le \varkappa^{-1} \varepsilon$. Applying~\eqref{Fourier_ident_1} and taking into account Remark~\ref{phi_l_bound_remark}, we have
	\begin{equation*}
		\begin{multlined}[c][0.9\textwidth]
			\Bigl\| \Bigl(e^{-i \tau \varepsilon^{-2} \mathcal{A}} e^{i  \left\langle \mathbf{k}^\circ, \cdot \right\rangle } \varsigma_j  - e^{i \left\langle \mathbf{k}^\circ, \cdot \right\rangle } \sum_{l=1}^{n} \varsigma_l  \tilde{J}_l e^{-i \tau \varepsilon^{-2} \mathcal{A}^{\mathrm{eff}}} J_j  \Bigr) \varepsilon^{3} (- \Delta + \varepsilon^{2}I)^{-3/2} (I - F_\varkappa) \Bigr\|_{L_2 (\mathbb{R}^d) \to L_2(\mathbb{R}^d)} \\ \le \biggl( \|\varsigma_j\|_{L_\infty} + \sum_{l=1}^n \|\varsigma_l\|_{L_\infty} \biggr)  \varkappa^{-1} \varepsilon.
		\end{multlined}
	\end{equation*}
	
	Consider the operator 
	$$
	\Bigl(e^{-i \tau \varepsilon^{-2} \mathcal{A}} e^{i  \left\langle \mathbf{k}^\circ, \cdot \right\rangle } \varsigma_j  - e^{i \left\langle \mathbf{k}^\circ, \cdot \right\rangle } \sum_{l=1}^{n} \varsigma_l \tilde{J}_l e^{-i \tau \varepsilon^{-2} \mathcal{A}^{\mathrm{eff}}} J_j \Bigr) \varepsilon^{3} (- \Delta + \varepsilon^{2}I)^{-3/2}  F_\varkappa,
	$$
	which, by virtue of~\eqref{Fourier_ident_1}, \eqref{Fourier_ident_2}, can be written as
	\begin{equation*}
		\Bigl( e^{-i \tau \varepsilon^{-2} \mathcal{A}}  \varsigma_j  - \sum_{l=1}^{n} \varsigma_l \Phi^* \tilde{J}_l e^{-i \tau \varepsilon^{-2} \mathfrak{g} (\delta \mathbf{k})} J_j \Phi \Bigr) \Phi^* \varepsilon^3 (|\delta \mathbf{k}|^2 + \varepsilon^2)^{-3/2} \chi_{\mathrm{B}_\varkappa(\mathbf{0})}(\delta \mathbf{k})\Phi e^{i\left\langle \mathbf{k}^\circ, \mathbf{x}\right\rangle }.
	\end{equation*}
	Recall that the operator $\mathcal{A}$ is decomposed into direct integral~\eqref{Gelfand_A_decompose}. Using~\eqref{Gelfand_Fourier} and~\eqref{Gelfand_Fourier_inv} (with $\mathbf{k}' = \mathbf{k}^\circ$), one obtains
	\begin{multline*}
		\Bigl\| \Bigl( e^{-i \tau \varepsilon^{-2} \mathcal{A}}  \varsigma_j  - \sum_{l=1}^{n} \varsigma_l \Phi^* \tilde{J}_l e^{-i \tau \varepsilon^{-2} \mathfrak{g} (\delta \mathbf{k})} J_j  \Phi \Bigr) \Phi^* \varepsilon^3 (|\delta \mathbf{k}|^2 + \varepsilon^2)^{-3/2} \chi_{\mathrm{B}_\varkappa(\mathbf{0})}(\delta \mathbf{k})\Phi e^{i\left\langle \mathbf{k}^\circ, \mathbf{x}\right\rangle } \Bigr\|_{L_2(\mathbb{R}^d) \to L_2(\mathbb{R}^d)}
		\\
		= \underset{\mathbf{k} \in \mathbb{T}^d}{\operatorname{ess-sup}\,}
		\Bigl\|	\Bigl( e^{-i \tau \varepsilon^{-2} \mathcal{A}(\mathbf{k})} \varsigma_j  - \sum_{l=1}^{n} \varsigma_l \tilde{J}_l e^{-i \tau \varepsilon^{-2} \mathfrak{g} (\delta \mathbf{k})} J_j \Bigr)  \varepsilon^3 (|\delta \mathbf{k}|^2 + \varepsilon^2)^{-3/2} \chi_{\mathrm{B}_\varkappa(\mathbf{0})}(\delta \mathbf{k}) P_0 \Bigr\|_{L_2(\Omega) \to L_2(\Omega)}.
	\end{multline*}
		
	From the inclusion $\Ran \varsigma_j P_0 \subset \mathfrak{N}$ and~\eqref{exp_basis} it is seen that
	\begin{multline*}
		\Bigl( e^{-i \tau \varepsilon^{-2} \mathcal{A}(\mathbf{k})} \varsigma_j  - \sum_{l=1}^{n} \varsigma_l \tilde{J}_l e^{-i \tau \varepsilon^{-2} \mathfrak{g} (\delta \mathbf{k})} J_j \Bigr)  \varepsilon^3 (|\delta \mathbf{k}|^2 + \varepsilon^2)^{-3/2} \chi_{\mathrm{B}_\varkappa(\mathbf{0})}(\delta \mathbf{k}) P_0
		\\
		= \Bigl(e^{-i \tau \varepsilon^{-2} \mathcal{A}(\mathbf{k})} - e^{-i \tau \varepsilon^{-2} \mathfrak{G}^\circ(\delta \mathbf{k}) P}\Bigr) P \varsigma_j \varepsilon^3 (|\delta \mathbf{k}|^2 + \varepsilon^2)^{-3/2} \chi_{\mathrm{B}_\varkappa(\mathbf{0})}(\delta \mathbf{k}) P_0.
	\end{multline*}
	Application of Theorem~\ref{exp_main_thrm} (with $\tau$ replaced by $\tau \varepsilon^{-2}$) together with the equality $\| \varsigma_j \|_{L_2(\Omega)}=1$ gives the estimate with the constants that do not depend on $\mathbf{k}$:
 	\begin{multline*}
		\bigl\| \bigl(e^{-i \tau \varepsilon^{-2} \mathcal{A}(\mathbf{k})} - e^{-i \tau \varepsilon^{-2} \mathfrak{G}^\circ(\delta \mathbf{k}) P}\bigr) P \varsigma_j \varepsilon^3 (|\delta \mathbf{k}|^2 + \varepsilon^2)^{-3/2} \chi_{\mathrm{B}_\varkappa(\mathbf{0})}(\delta \mathbf{k}) P_0 \bigr\|_{L_2(\Omega) \to L_2(\Omega)} \\
		\le \bigl(3C_7 |\delta \mathbf{k}| + C_{11} \varepsilon^{-2} |\tau| |\delta \mathbf{k}|^3\bigr) \varepsilon^3 (|\delta \mathbf{k}|^2 + \varepsilon^2)^{-3/2} \le (3C_7 + C_{11} |\tau|) \varepsilon.
	\end{multline*} 
	This completes the proof.
\end{proof}

\begin{remark}
	Interpolating between~\eqref{thrm_est_operators} and the estimate
	\begin{equation*}
		\Bigl\| e^{-i \tau \mathcal{A}_\varepsilon} e^{i \varepsilon^{-1} \left\langle \mathbf{k}^\circ, \cdot \right\rangle } \varsigma_j^\varepsilon  - e^{i \varepsilon^{-1} \left\langle \mathbf{k}^\circ, \cdot \right\rangle } \sum_{l=1}^{n} \varsigma_l^\varepsilon \tilde{J}_l e^{-i \tau \mathcal{A}^{\mathrm{eff}}_\varepsilon} J_j  \Bigr\|_{L_2(\mathbb{R}^d) \to L_2(\mathbb{R}^d)} \le \|\varsigma_j\|_{L_\infty} + \sum_{l=1}^n \|\varsigma_l\|_{L_\infty},
	\end{equation*}
	we obtain
	\begin{multline*}
		\Bigl\| e^{-i \tau \mathcal{A}_\varepsilon} e^{i \varepsilon^{-1} \left\langle \mathbf{k}^\circ, \cdot \right\rangle } \varsigma_j^\varepsilon  - e^{i \varepsilon^{-1} \left\langle \mathbf{k}^\circ, \cdot \right\rangle } \sum_{l=1}^{n} \varsigma_l^\varepsilon \tilde{J}_l e^{-i \tau \mathcal{A}^{\mathrm{eff}}_\varepsilon} J_j  \Bigr\|_{H^q(\mathbb{R}^d) \to L_2(\mathbb{R}^d)} \\ \le  \biggl( \|\varsigma_j\|_{L_\infty} + \sum_{l=1}^n \|\varsigma_l\|_{L_\infty} \biggr)^{1-q/3} \mathcal{C}^{q/3} (1+|\tau|)^{q/3} \varepsilon^{q/3}, \qquad 0 \le q \le 3, \quad j = 1,\ldots,n.
	\end{multline*}
	When applied to Cauchy problem~\eqref{Cauchy_problem}, this leads to the estimate
	\begin{multline*}
		\| u_{j,\varepsilon} (\cdot, \tau) - u^{\mathrm{eff}}_{j,\varepsilon} (\cdot, \tau) \|_{L_2(\mathbb{R}^d)} \le \biggl( \|\varsigma_j\|_{L_\infty} + \sum_{l=1}^n \|\varsigma_l\|_{L_\infty} \biggr)^{1-q/3} \mathcal{C}^{q/3} (1+|\tau|)^{q/3} \varepsilon^{q/3} \|f_j\|_{H^q(\mathbb{R}^d)}, \\ j = 1, \ldots,n,
	\end{multline*}
	for $\varepsilon > 0$, $\tau \in \mathbb{R}$, $f_j \in H^q(\mathbb{R}^d)$, and $0 \le q \le 3$.
\end{remark}

\begin{remark}
	Suppose that $\lambda_0$ is a spectral edge of the operator $\mathcal{A}$ and the corresponding extremal value is attained by one band function\textup{:} $\lambda_0 = E_s(\mathbf{k}^\circ)$, $\lambda_0 \ne E_l(\mathbf{k}^\circ)$, $l \ne s$. In this case, $n = 1$, and, by~\eqref{frakG1=0_cond}, \eqref{frak_g_1}, $\mathfrak{g}^{1,11} = 0$. Let $u_\varepsilon (\mathbf{x},\tau) \coloneqq u_{1,\varepsilon} (\mathbf{x},\tau)$ be the solution of problem~\eqref{Cauchy_problem}, and let $\tilde{v}^{\mathrm{eff}} (\mathbf{x},\tau)$ be the solution of the problem
	\begin{equation*}
		\left\{
		\begin{aligned}
			&i \frac{\partial}{\partial \tau} \tilde{v}^{\mathrm{eff}} (\mathbf{x},\tau) = \widetilde{\mathcal{A}}^{\mathrm{eff}} \tilde{v}^{\mathrm{eff}} (\mathbf{x},\tau) ,\\
			&\tilde{v}^{\mathrm{eff}} (\mathbf{x},0) = f_1 (\mathbf{x}),
		\end{aligned}
		\right.
	\end{equation*}
	where $\widetilde{\mathcal{A}}^{\mathrm{eff}} = - \operatorname{div} \mathfrak{g}^{2,11} \nabla$. Let $\varepsilon > 0$, $\tau \in \mathbb{R}$, $f_1 \in H^3(\mathbb{R}^d)$. We have
	\begin{equation*}
		\| u_\varepsilon (\cdot, \tau) - e^{- i \tau \varepsilon^{-2} \lambda_0} e^{i \varepsilon^{-1} \left\langle \mathbf{k}^\circ, \mathbf{x} \right\rangle} \varsigma_1^\varepsilon (\mathbf{x})  \tilde{v}^{\mathrm{eff}} (\cdot, \tau) \|_{L_2(\mathbb{R}^d)} \le \mathcal{C} (1+|\tau|) \varepsilon \|f_1\|_{H^3(\mathbb{R}^d)}.
	\end{equation*}
\end{remark}